\documentclass{amsart}
\usepackage{amsmath,amsfonts,amssymb,amsthm,amscd}
\usepackage[small,nohug,heads=littlevee]{diagrams}

\begin{document}

\theoremstyle{definition}
\newtheorem{step}{Step}
\newtheorem{definition}{Definition}[section]
\theoremstyle{remark}
\newtheorem{remark}{Remark}[section]
\theoremstyle{plain}
\newtheorem{prop}{Proposition}[section]
\newtheorem{claim}{Claim}[section]
\newtheorem{lemma}{Lemma}[section]
\newtheorem{theorem}{Theorem}[section]
\newtheorem{cor}{Corollary}[section]

\title{Additivity of Spin$^c$ Quantization under Cutting}
\author{Shay Fuchs}
\begin{abstract}A $G$-equivariant spin$^c$
structure on a manifold gives rise to a virtual
representation of the group $G$, called \emph{the
spin$^c$ quantization} of the manifold. We
present a cutting construction for
$S^1$-equivariant spin$^c$ manifolds, and show
that the quantization of the original manifold is
isomorphic to the direct sum of the quantizations
of the cut spaces. Our proof uses Kostant-type
formulas, which express the quantization in terms
of local data around the fixed point set of the
$S^1$-action.
\end{abstract}

\maketitle

\section{Introduction}
In this paper we discuss $S^1$-equivariant
spin$^c$ structures on compact oriented
Riemannian $S^1$-manifolds, and the Dirac
operator associated to those structures. The
index of the Dirac operator is a virtual
representation of $S^1$, and is called \emph{the
spin$^c$ quantization} of the spin$^c$ manifold.

 Also, we
describe a cutting construction for spin$^c$
structures. Cutting was first developed by E.
Lerman for symplectic manifolds (see \cite{L}),
and then extended to manifolds that posses other
structures. In particular, our recipe is closely
related to the one described in \cite{CKT}.

The goal of this paper is to point out a relation
between spin$^c$ quantization and cutting. We
claim that the quantization of our original
manifold is isomorphic (as a virtual
representation) to the direct sum of the
quantizations of the cut spaces. We refer to this
property as `additivity under cutting'.

In \cite{GSW}, Guillemin, Sternberg and Weitsman
define \emph{signature quantization} and show
that it satisfies `additivity under cutting'. In
fact, this observation motivated the present
paper.

It is important to mention that in the
 this property does not hold for the most common
 `almost-complex quantization'.
In this case, we start with an almost complex
compact manifold, and a Hermitian line bundle
with Hermitian connection, and construct the
Dolbeaut-Dirac operator associated to this data.
Its index is a virtual vector space, and in the
presence of an $S^1$-action on the manifold and
the line bundle, we get a virtual representation
of $S^1$, called the Dolbeau-Dirac quantization
of the manifold (see \cite{Kar} or
\cite{Symp-surj}). This is a special case of our
spin$^c$ quantization, since an almost complex
structure and a complex line bundle determine a
spin$^c$ structure, which gives rise to the same
Dirac operator (See Lemma 2.7 and Remark 2.9 in
\cite{CKT}, and Appendix D in \cite{LM}).
However, in the almost complex case, the cutting
is done along the zero level set of the moment
map determined by the line bundle and the
connection. This results in additivity for all
weights except zero. More precisely, if
$N_\pm(\mu)$ denotes the multiplicity of the
weight $\mu$ in the almost complex quantization
of the cut spaces, and $N(\mu)$ is the weight of
$\mu$ in the quantization of the original
manifold, we have (see p.258 in \cite{Symp-surj})
$$ N(\mu)=N_+(\mu)+N_-(\mu)\qquad,\qquad \mu\ne
0$$ but $$ N(0)=N_+(0)=N_-(0)$$ and therefore
there is no additivity in general.

On the other hand, if spin$^c$ cutting is done
for a spin$^c$ manifold $M$ (in particular, the
spin$^c$ structure can come from an almost
complex structure), then the additivity will hold
for any weight. Roughly speaking, this happens
because the spin$^c$ cutting is done at the level
set $1/2$ of the `moment map', which is not a
weight (i.e., an integer) for the group $S^1$.

In order to make this paper as self-contained as
possible, we review the necessary background on
spin$^c$ equivariant structures, Clifford
algebras and spin$^c$ quantization in Section
\ref{spin-c-quant}. We describe in details the
cutting process in Section \ref{spin-c-cut}. In
Sections \ref{iso} and \ref{noniso-sec} we
develop Kostant-type formulas forspin$^c$
quantizations in terms of local data around
connected components of the fixed point set, and
finally in Section \ref{additivity} we prove the
additivity result. In Section \ref{two-sphere},
we give a detailed example that illustrates the
additivity property of spin$^c$ quantization. In
particular, we classify and cut all the
$S^1$-equivariant spin$^c$ structures on the
two-sphere. In the last section, we comment about
the relation of our work to the original
symplectic cutting construction.

Throughout this paper, all spaces will assumed to
be smooth manifolds, and all maps and actions are
assumed to be smooth. The principal action in a
principal bundle will be always a right action. A
real vector bundle $E$, equipped with a fiberwise
inner product, will be called a \emph{Riemannian
vector bundle}. If the fibers are also oriented,
then its bundle of oriented orthonormal frames
will be denoted by $SOF(E)$. For an oriented
Riemannian manifold $M$, we will simply write
$SOF(M)$, instead of $SOF(TM)$.

\textbf{Acknowledgements.} I would like to thank
my supervisor, Yael Karshon, for offering me this
project, guiding and supporting me through the
process of developing and writing the material,
and for having always good advice and a lot of
patience. I also would like to thank Lisa Jeffrey
and Eckhard Meinrenken for useful discussions and
important comments.

\section{Spin$^c$ Quantization}\label{spin-c-quant}
In this section we define the concept of
\emph{spin$^c$ quantization} as the index of the
Dirac spin$^c$ operator associated to a manifold
endowed with a spin$^c$ structure. The
quantization will be a virtual complex vector
space, and in the presence of a Lie group action
it will be a virtual representation of that
group.

\subsection{Spin$^c$
structures}

\begin{definition} Let  $V$  be a finite dimensional vector space over  $\mathbb{K}=\mathbb{R}\mbox{ or }
\mathbb{C}$,  equipped with a symmetric bilinear
form  $B:V\times V\rightarrow\mathbb{K}$.  The
\emph{Clifford algebra} \ $Cl(V,B)$  is the
quotient  $T(V)/I(V,B)$  where  $T(V)$  is the
tensor algebra of  $V$  and  $I(V,B)$  is the
ideal generated by $\{ v\otimes v-B(v,v)\cdot
1\;:\; v\in V\}$.
\end{definition}

\begin{remark}
If $v_1,\dots,v_n$ is an orthogonal basis for
$V$, then $Cl(V,B)$ is the algebra generated by
$v_1,\dots,v_n$ subject to the relations
$v_i^2=B(v_i,v_i)\cdot 1$ and $v_i v_j=-v_j v_i$ for $i\neq j$.\\
Note that $V$ is a vector subspace of $Cl(V,B)$.
\end{remark}

\begin{definition}
If $V=\mathbb{R}^k$ and $B$ is minus the standard inner product on
$V$, then define the following objects:
\begin{enumerate}
\item $C_k:=Cl(V,B)$, and $C_k^c:=Cl(V,B)\otimes\mathbb{C}$.\\ These are finite dimensional algebras over $\mathbb{R}$ and $\mathbb{C}$, respectively.
\item The \emph{spin group} $$Spin(k)=\{v_1 v_2 \dots v_l\;:\; v_i\in\mathbb{R}^k,\ ||v_i||=1
\mbox{ and } 0\le l \mbox{ is even}\}\subset C_k$$
\item The \emph{spin$^c$ group} $$Spin^c(k)= {\left(Spin(k)\times U(1)\right)}\diagup{K}$$ where $U(1)\subset\mathbb{C}$ is the unit circle and
$K=\{(1,1),(-1,-1)\}$.
\end{enumerate}

\end{definition}

\begin{remark} \
\begin{enumerate}
\item Equivalently, one can define
\begin{multline*}
\qquad \ Spin^c(k)=\\ =\left\{c\cdot v_1\cdots v_l\;: \linebreak \;
v_i\in\mathbb{R}^k,\ ||v_i||=1,\  0\le l \mbox{ is even, }\mbox{ and
}  c\in U(1)\right\}\subset C^c_k
\end{multline*}
\item The group $Spin(k)$ is connected for $k\ge 2$.
\end{enumerate}
\end{remark}

\begin{prop}
\

\begin{enumerate}
\item There is a linear map $C_k\rightarrow C_k\;,\; x\mapsto x^t$, characterized by
$(v_1\dots v_l)^t=v_l\dots v_1$ for all
$v_1,\dots,v_l\in\mathbb{R}^k$.
\item For each $x\in Spin(k)$ and $y\in\mathbb{R}^k$, we have
$xyx^t\in\mathbb{R}^k$.
\item For each $x\in Spin(k)$, the map
$\lambda(x):\mathbb{R}^k\rightarrow\mathbb{R}^k\;,\;
y\mapsto xyx^t$, is in $SO(k)$, and
$\lambda:Spin(k)\rightarrow SO(k)$ is a double
covering for $k\ge 1$. It is a universal covering
map for $k\ge 3$.
\end{enumerate}
\end{prop}

For the proof, see page 16 in \cite{Fr}.

\begin{definition}
Let $M$ be a manifold and $Q$ a principal
$SO(k)$-bundle on $M$. A \emph{spin$^c$
structure} on $Q$ is a principal
$Spin^c(k)$-bundle $P\rightarrow M$, together
with a map $\Lambda:P\rightarrow Q$, such that
the following diagram commutes.

$$
\begin{CD}
P\times Spin^c(k)     @>>>          P  \\
@VV\Lambda\times\lambda^c V   @VV\Lambda V \\
Q\times SO(k)       @>>>       Q\\
\end{CD}
$$\\

Here, the maps corresponding to the horizontal arrows are the
principal actions, and $\lambda^c:Spin^c(k)\rightarrow SO(k)$ is
given by $[x,z]\mapsto\lambda(x)$, where $\lambda:Spin(k)\rightarrow
SO(k)$ is the double covering.
\end{definition}

\begin{remark}\
\begin{enumerate}
\item  A spin$^c$ structure on an oriented Riemannian vector bundle $E$ is a spin$^c$ structure on the associated bundle of
oriented orthonormal frames,  $SOF(E)$.
\item A spin$^c$ structure on an oriented Riemannian manifold is a
spin$^c$ structure on its tangent bundle.
\item Given a spin$^c$ structure on $Q\rightarrow M$, \emph{its determinant line bundle}
is $\mathbb{L}=P\times_{Spin^c(k)}\mathbb{C}$,
where the left action of $Spin^c(k)$ on
$\mathbb{C}$ is given by $[x,z]\cdot w=z^2 w$.
This is a hermitian line bundle over $M$.
\end{enumerate}
\end{remark}

\subsection{Equivariant spin$^c$ structures}

\begin{definition}
Let $G,H$ be Lie groups. A \emph{$G$-equivariant
principal $H$-bundle} is a principal $H$-bundle
$\pi:Q\rightarrow M$ together with left
$G$-actions on $Q$ and $M$, such that
\begin{enumerate}
\item $\pi(g\cdot q)=g\cdot\pi(q)$ for all $g\in G\;,\; q\in Q$\\
(i.e., $G$ acts on the fiber bundle $\pi:Q\rightarrow M$).
\item $(g\cdot q)\cdot h=g\cdot(q\cdot h)$ for all $g\in G\;,\; q\in Q\;,\; h\in
H$\\
(i.e., the actions of $G$ and $H$ commute).
\end{enumerate}

\begin{remark}
It is convenient to think of a $G$-equivariant principal $H$-bundle
in terms of the following commuting diagram (the horizontal arrows correspond to the $G$ and $H$ actions).\\
$$
\begin{CD}
G\times Q       @>>>     Q    @<<<     Q\times H\\
@VId\times\pi VV                   @VV\pi V           @.\\
G\times M       @>>>     M    @.\\
\end{CD}
$$\\
\end{remark}

\begin{definition}
Let $\pi:E\rightarrow M$ be a fiberwise  oriented
Riemannian vector bundle, and let $G$ be a Lie
group. If a  $G$-action on $E\to M$ is given that
preserves the orientations and the inner products
of the fibers, we will call $E$ a
\emph{$G$-equivariant oriented Riemannian vector
bundle}.
\end{definition}

\begin{remark} \
\begin{enumerate}
\item If $E$ is a $G$-equivariant oriented Riemannian vector bundle, then $SOF(E)$ is a $G$-equivariant
principal $SO(k)$-bundle, where $k=rank(E)$.
\item If a Lie group G acts on an oriented Riemannian manifold $M$ by orientation preserving
isometries, then the frame bundle $SOF(M)$ becomes a $G$-equivariant
principal $SO(m)$-bundle, where $m=$dim$(M)$.
\end{enumerate}
\end{remark}

\end{definition}

\begin{definition}
Let $\pi:Q\rightarrow M$ be a $G$-equivariant principal
$SO(k)$-bundle. \emph{A $G$-equivariant spin$^c$ structure} on $Q$
is a spin$^c$ structure $\Lambda:P\rightarrow Q$ on $Q$, together
with a a left action of $G$ on $P$, such that
\begin{enumerate}
\item $\Lambda(g\cdot p)=g\cdot\Lambda(p)$ for all $p\in P$, $g\in
G$ (i.e., $G$ acts on the bundle $P\rightarrow Q$).
\item $g\cdot(p\cdot x)=(g\cdot p)\cdot x$ for all $g\in G$, $p\in P$, $x\in
Spin(k)$\\ (i.e., the actions of $G$ and $Spin^c(k)$ on $P$
commute).
\end{enumerate}
\end{definition}

\begin{remark}\label{Remark_spin-c_str} \
\begin{enumerate}
\item We have the following commuting diagram (where the
horizontal arrows correspond to the principal and the $G$-actions).

$$
\begin{CD}
G\times P       @>>>     P    @<<<     P\times Spin^c(k)\\
@V Id\times\Lambda VV                    @V\Lambda VV          @V\Lambda\times\lambda^c VV\\
G\times Q       @>>>     Q    @<<<     Q\times SO(k)\\
@V Id\times\pi VV                   @V\pi VV           @.\\
G\times M       @>>>     M    @.\\
\end{CD}
$$\\
\item The bundle $P\rightarrow M$ is a $G$-equivariant
principal $Spin^c(k)$-bundle.
\item The
determinant line bundle
$\mathbb{L}=P\times_{Spin^c(k)}\mathbb{C}$ is a
$G$-equivariant Hermitian line bundle.
\end{enumerate}
\end{remark}

\subsection{Clifford multiplication and spinor bundles}

\begin{prop}
The number of inequivalent irreducible (complex)
representations of the algebra
$C^c_k=C_k\otimes\mathbb{C}$ is $1$ if $k$ is
even and $2$ if $k$ is odd.
\end{prop}

For a proof, see Theorem I.5.7 in \cite{LM}.\\

Note that, for all $k$, $\mathbb{R}^k\subset
C_k\subset C^c_k$.

\begin{definition}
Let $k$ be a positive integer. Define a
\emph{Clifford multiplication map}
$$\mu:\mathbb{R}^k\otimes\Delta_k\to\Delta_k\qquad\text{by}\quad\mu(x\otimes
v)=\rho_k(x)v$$ where $\rho_k:C_k^c\to
End(\Delta_k)$ is an irreducible representation
of $C_k^c$ (a choice is to be made if $k$ is
odd).
\end{definition}

\begin{definition}
Let $k$ be a positive integer and $\rho_k$  an
irreducible representation of $C_k^c$. The
restriction of $\rho_k$ to the group
$Spin(k)\subset C_k\subset C_k^c$ is called
\emph{the complex spin representation} of
$Spin(k)$. It will be also denoted by $\rho_k$.
\end{definition}

\begin{remark} \
For an odd integer $k$, the complex spin
representation is independent of the choice of an
irreducible representation of $C_k^c$ (see
Proposition I.5.15 in \cite{LM}).
\end{remark}

The following proposition summarizes a few facts
about the complex spin representation. Proofs can
be found in \cite{Fr} and in \cite{LM}.
\begin{prop}\label{T:spinrep} Let $\rho_k:Spin(k)\rightarrow
End(\Delta_k)$ be the complex spin
representation. Then
\begin{enumerate}
\item $dim_\mathbb{C}\Delta_k=2^l$, where $l=k/2$ if $k$ is even, and
$l=(k-1)/2$ if $k$ is odd.

\item $\rho_k$ is a faithful representation of $Spin(k)$.

\item If $k$ is odd, then $\rho_k$ is
irreducible.

\item If $k$ is even, then $\rho_k$ is
reducible, and splits as a sum of two
inequivalent irreducible representations of the
same dimension,
$$\rho_k^+:Spin(k)\rightarrow End(\Delta^+_k)\qquad\text{and}\qquad
\rho_k^-:Spin(k)\rightarrow End(\Delta^-_k)\ .$$

\end{enumerate}
\end{prop}

\begin{remark}
The representation $\rho_k$ extends to a representation of the group
$Spin^c(k)$, and will be also denoted by $\rho_k$. Explicitly,
$$\rho_k:Spin^c(k)\rightarrow End(\Delta_k)\qquad,\qquad
\rho_k([x,z])v=z\cdot \rho_k(x)v\ .$$
\end{remark}

\begin{definition}
Let $P$ be a spin$^c$ structure on an oriented Riemannian manifold
$M$. Then the \emph{spinor bundle} of the spin$^c$ structure is the
complex vector bundle $S=P\times_{Spin^c(m)}\Delta_m$, where
$m=dim(M)$.\\
If $P$ is a $G$-equivariant spin$^c$ structure, then $S$ will be a
$G$-equivariant complex vector bundle.
\end{definition}

\begin{remark}
It is possible to choose a Hermitian inner
product on $\Delta_k$ which is preserved by the
action of the group $Spin^c(k)$. This induces a
Hermitian inner product on the spinor bundle. In
the $G$-equivariant case, $G$ will act on the
fibers of $S$ by Hermitian transformations.
\end{remark}

From Proposition \ref{T:spinrep} we get

\begin{prop}
Let $P$ be a ($G$-equivariant) spin$^c$ structure
on an oriented Riemannian manifold $M$ of even
dimension, and let $S$ be the corresponding
spinor bundle. Then $S$ splits as a sum
$S=S^+\oplus S^-$ of two ($G$-equivariant)
complex vector bundles.

\end{prop}

\begin{remark}
If $M$ is an oriented Riemannian manifold,
equipped with a spin$^c$ structure, and a
corresponding spinor bundle $S$, then a Clifford
multiplication map
$\mu:\mathbb{R}^k\otimes\Delta_k\to\Delta_k$
induces a map on the associated bundles
$TM\otimes S\to S$. This map is also called
Clifford multiplication and will be denoted by
$\mu$ as well.
\end{remark}

\subsection{The spin$^c$ Dirac operator}\ \\
The following is a reformulation of Proposition D.11 from \cite{LM}:

\begin{prop}
Let $M$ be an oriented Riemannian manifold of
dimension $m\ge 1$, $P\rightarrow SOF(M)$ a
spin$^c$ structure on $M$, and $P_1=P/Spin(m)$
(this quotient can be defined since $Spin(m)$
embeds naturally in $Spin^c(m)$). Then
\begin{enumerate}
\item $P_1$ is a principal $U(1)$-bundle over $M$, and $P\rightarrow
SOF(M)\times P_1$ is a double cover.

\item The determinant line bundle of the spin$^c$
structure is naturally isomorphic to
$\mathbb{L}=P_1\times_{U(1)}\mathbb{C}$.
\item If $A:TP_1\rightarrow i\mathbb{R}$ is an invariant connection,
and $Z:T(SOF(M))\rightarrow~\mathfrak{so}(m)$ the
Levi-Civita connection on $M$, then the
$SO(m)\times U(1)$-invariant connection $Z\times
A$ on $SOF(M)\times P_1$ lifts to a unique
$Spin^c(m)$-invariant connection on its double
cover $P$.
\end{enumerate}
\end{prop}

\begin{remark}
If $G$ acts on $M$ by orientation preserving
isometries, $P$ is a $G$-equivariant spin$^c$
structure on $M$, and the connection $A$ on $P_1$
is chosen to be $G$-invariant, then $Z\times A$
and its lift to $P$ will be $G$-invariant.
\end{remark}

\begin{definition}
Assume the following data is given:
\begin{enumerate}
\item An oriented Riemannian manifold $M$ of dimension $m$.
\item A spin$^c$ structure $P\rightarrow SOF(M)$ on $M$, with the associated spinor bundle~$S$.
\item A connection on $P_1=P/Spin(m)$ which gives rise to a
covariant derivative $\nabla:\Gamma(S)\rightarrow\Gamma(T^*M\otimes
S)$
\end{enumerate}

\noindent The \emph{Dirac spin$^c$ operator} (or
simply, the \emph{Dirac operator}) associated to
this data is the composition

$$D:\Gamma(S)\xrightarrow{\quad\nabla\quad}\Gamma(T^*M\otimes
S)\xrightarrow{\quad\simeq\quad}\Gamma(TM\otimes
S)\xrightarrow{\quad\mu\quad}\Gamma(S)\ ,$$

\noindent where the isomorphism is induced by the
Riemannian metric (which identifies $T^*M\simeq
TM$), and $\mu$ is the Clifford multiplication.
\end{definition}

\begin{remark} \
\begin{enumerate}
\item Since there are two ways to define $\mu$ when $k$ is odd, one
has to make a choice for $\mu$ to get a
well-defined Dirac operator.
\item If $G$ acts on $M$ by orientation preserving isometries, the
spin$^c$ structure on $M$ is $G$-equivariant, and the connection on
$P_1$ is $G$-invariant, then the Dirac operator $D$ will commute
with the $G$-action on $\Gamma(S)$.
\item If $dim(M)$ is even, then the Dirac operator decomposes into
a sum of two operators
$D^{\pm}:\Gamma(S^{\pm})\to\Gamma(S^{\mp})$
(since $\mu$ interchanges $S^+$ and $S^-$), which
are also called Dirac operators.
\item If the manifold $M$ is complete, then the Dirac operator
is essentially self-adjoint on $L^2(S)$, the
square integrable sections of $S$ (See Theorem
II.5.7 in \cite{LM} or chapter 4 in \cite{Fr}).
\end{enumerate}
\end{remark}

\subsection{Spin$^c$ quantization}\label{spin^c quantization}

We now restrict to the case of an even
dimensional oriented Riemannian manifold $M$
which is also compact. Since the concept of
spin$^c$ quantization will be defined as the
index of the operator $D^+$, it makes sense to
define it only for even dimensional manifolds.
The compactness is used to ensure that
$dim(ker(D^+))$ and $dim(coker(D^+))$ are finite.

\begin{definition}
Assume that the following data is given:
\begin{enumerate}
\item An oriented compact Riemannian manifold $M$ of dimension
$2m$.
\item $G$ a Lie group that acts on $M$ by orientation preserving
isometries.
\item $P\to SOF(M)$ a $G$-equivariant spin$^c$ structure.
\item A $U(1)$-invariant connection on $P_1=P/Spin(2m)$.
\end{enumerate}

\noindent Then the \emph{spin$^c$ quantization of $M$}, with respect
to the above date, is the virtual complex $G$-representation
$Q(M)=ker(D^+)-coker(D^+)$.\\
The \emph{index} of $D^+$ is the integer
$index(D^+)=dim(ker(D^+))-dim(coker(D^+))$.
\end{definition}

\begin{remark}
In the absence of a $G$ action, the spin$^c$ quantization is just a
virtual complex vector space.
\end{remark}

\section{Spin$^c$ cutting}\label{spin-c-cut}
In \cite{L} Lerman describes the symplectic
cutting construction for symplectic manifolds
equipped with a Hamiltonian $G$-action. In
\cite{CKT} this construction is generalized to
manifolds with other structures, including
spin$^c$ manifolds. However, the cutting of a
spin$^c$ structure is incomplete in \cite{CKT},
since it only produces a spin$^c$ principal
bundle on the cut spaces $P_{cut}\to M_{cut}$,
without constructing a map $P_{cut}\to
SOF(M_{cut})$.

In this section, we describe the construction
from section 6 in \cite{CKT}  and fill the necessary gaps.\\
From now on we will work with $G$-equivariant
spin$^c$ structures. This includes the
non-equivariant case when $G$ is taken to be the
trivial group $\{e\}$.

\subsection{The product of two spin$^c$ structures}\label{product}

Note that the group $SO(m)\times SO(n)$ naturally
embeds in $SO(n+m)$ as block matrices, and
therefore it acts on $SO(n+m)$ from the left by
left multiplication.

The proof of the following claim is straightforward.

\begin{claim}
Let $M$ and $N$ be two oriented Riemannian
manifolds of respective dimensions $m$ and $n$.
Then the map $$(SOF(M)\times
SOF(N))\times_{SO(m)\times SO(n)} SO(n+m)\to
SOF(M\times N)$$
$$ [(f,g),K]\mapsto (f,g)\circ K$$
is an isomorphism of principal $SO(n+m)$-bundles.\\
Here, $f:\mathbb{R}^m\xrightarrow{\sim}T_aM$ and
$g:\mathbb{R}^n\xrightarrow{\sim}T_bN$ are frames, and
$K:\mathbb{R}^{m+n}\to\mathbb{R}^{m+n}$ is in $SO(m+n)$.
\end{claim}

The above claim suggests a way to define the
product of two spin$^c$ manifolds (see also Lemma
6.10 from \cite{CKT}). There is a natural group
homomorphism $j:Spin(m)\times Spin(n)\to
Spin(m+n)$, which is induced from the embeddings
$$\mathbb{R}^m\hookrightarrow\mathbb{R}^m\times\{0\}\subset\mathbb{R}^{m+n}
\qquad\text{and}\qquad
\mathbb{R}^n\hookrightarrow\{0\}\times\mathbb{R}^n\subset\mathbb{R}^{m+n}\
.$$ This gives rise to a homomorphism
$$j^c:Spin^c(m)\times Spin^c(n)\to
Spin^c(m+n)\quad,\quad([A,a],[N,b])\mapsto
[j(A,B),ab]\ ,$$ and therefore $Spin^c(m)\times
Spin^c(n)$ acts from the left on $Spin^c(m+n)$
via $j^c$.

If a group $G$ acts on two manifolds $M$ and $N$, then it clearly
acts on $M\times N$ by $g\cdot (m,n)=(g\cdot m, g\cdot n)$, and the
above claim generalizes to this case as well.

\begin{definition}
Let $G$ be a Lie group that acts on two oriented Riemannian
manifolds $M$,$N$ by orientation preserving isometries. Let $P_M\to
SOF(M)$ and $P_N\to SOF(N)$ be $G$-equivariant spin$^c$ structures
on $M$ and $N$. Then $$P=(P_M\times P_N)\times_{Spin^c(m)\times
Spin^c(n)} Spin^c(m+n)\to SOF(M\times N)$$ is a $G$-equivariant
spin$^c$ structure on $M\times N$, called \emph{the product} of the
two given spin$^c$ structures.
\end{definition}

\begin{remark}
In the above setting, if $L_M$ and $L_N$ are the determinant line
bundles of the spin$^c$ structures on $M$ and $N$, respectively,
then the determinant line bundle of $P\to SOF(M\times N)$ is
$L_M\boxtimes L_N$ (exterior tensor product). See Lemma 6.10 from
\cite{CKT} for details.
\end{remark}

\subsection{Restriction of a spin$^c$ structure}\label{rest}
In general, it is not clear how to restrict a spin$^c$ structure
from a Riemannian oriented manifold to a submanifold. However, for
our purposes, it suffices to work with co-oriented submanifolds of co-dimension 1.\\

The proof of the following claim is
straightforward.
\begin{claim}
Assume that the following data is given:
\begin{enumerate}
\item $M$ an oriented Riemannian manifold of dimension $m$.
\item $G$ a Lie group that acts on $M$ by orientation preserving
isometries.
\item $Z\subset M$ a $G$-invariant co-oriented
submanifold of co-dimension 1.
\item $P\to SOF(M)$ a $G$-equivariant spin$^c$ structure on $M$.
\end{enumerate}
Define an injective  map
$$i:SOF(Z)\to SOF(M)\qquad,\qquad i(f)(a_1,\dots,a_m)=f(a_1,\dots,a_{m-1})+a_m\cdot v_p$$
where $f:\mathbb{R}^{m-1}\xrightarrow{\sim}T_p Z$
is a frame in $SOF(Z)$, and $v\in\Gamma(TM|_Z)$
is the vector field of positive unit
vectors, orthogonal to $TZ$.\\
Then the pullback $P'=i^*(P)\to SOF(Z)$ is a $G$-equivariant
spin$^c$ structure on $Z$, called \emph{the restriction of $P$ to
$Z$}.
\end{claim}

\begin{remark} \
\begin{enumerate}
\item This is the relevant commutative diagram for the claim:

$$
\begin{CD}
P'=i^*(P) @>>> P \\
@VVV @VVV \\
SOF(Z) @>i>> SOF(M)\\
@VVV @VVV\\
Z @>>> M\\
\end{CD}
$$\\

\item The principal action of $Spin^c(m-1)$ on $P'$ is obtained
using the natural inclusion $Spin^c(m-1)\hookrightarrow Spin^c(m)$.

\item The determinant line bundle of $P'$ is the restriction to $Z$ of the
determinant line bundle of $P$.
\end{enumerate}
\end{remark}

\subsection{Quotients of spin$^c$ structures}\label{quot-spin-c}
We now discuss the process of taking quotients of
a spin$^c$ structure with respect to a group
action. Since the basic cutting construction
involves an $S^1$-action, we will only deal with
circle
actions.\\

\noindent Assume that the following data is given:
\begin{enumerate}
\item An oriented Riemannian manifold $Z$ of dimension $n$.
\item A free action $S^1\circlearrowright Z$ by isometries.
\item $P\to SOF(Z)$ an $S^1$-equivariant spin$^c$ structure on
$Z$.\\
\end{enumerate}

\noindent Denote by $\frac{\partial}{\partial\theta}\in Lie(S^1)$ an
infinitesimal generator, by
$\left(\frac{\partial}{\partial\theta}\right)_Z\in\chi(Z)$ the
corresponding vector field, and by $\pi:Z\to Z/S^1$ the quotient
map. Also let $V=\pi^*\left(T\left(Z/S^1\right)\right)$. This is an $S^1$-equivariant vector bundle over $Z$.
$$
\begin{CD}
V=\pi^*\left(T\left(Z/S^1\right)\right) @>>>  T\left(Z/S^1\right) \\
@VVV @VVV \\
Z @>\pi>> Z/S^1\\
\end{CD}
$$

We have the following simple fact.

\begin{lemma}
The map $$\left(\left(\frac{\partial}{\partial\theta}\right)_Z\right)^\bot
\longrightarrow V\qquad
\qquad v\in T_p Z\longmapsto (p,\pi_*
v)\in V_p$$ is an isomorphism of
$S^1$-equivariant vector bundles over $Z$.
\end{lemma}

\begin{remark}\label{quotients}
Using this lemma, we can endow $V$ with a
Riemannian metric and orientation, and hence $V$
becomes an oriented Riemannian vector bundle (of
rank $n-1$). We will think of $V$  as a
sub-bundle of $TZ$.

Also, if an orthonormal frame in $V$ is chosen,
then its image in $T(Z/S^1)$ is declared to be
orthonormal. This endows $Z/S^1$ with an
orientation and a Riemannian metric, and hence it
makes sense to speak of $SOF(Z/S^1)$.
\end{remark}

Now define a map $\eta\colon SOF(V)\to SOF(Z)$ in the following
way.
If $f:\mathbb{R}^{n-1}\xrightarrow{\simeq}V_p$ is a frame,
then $\eta(f)\colon\mathbb{R}^n\to T_pZ$ will be given by
 $\eta(f)e_i=f(e_i)$ for $i=1,\dots,n-1$ and
 $\eta(f)e_n$ is a unit vector in the direction of $\left(\frac{\partial}{\partial\theta}\right)_{Z,p}$ .

The following lemmas are used to get a spin$^c$
structure on $Z/S^1$. Their proofs are
straightforward and left to the reader.

\begin{lemma}
The pullback $\eta^*(P)\subset SOF(V)\times P$ is an
$S^1$-equivariant spin$^c$ structure on $SOF(V)$.\\
(The $S^1$-action on $\eta^*(P)$ is induced from
the $S^1$-actions on $SOF(V)$ and $P$, and the
right action of $Spin^c(n-1)$ is induced by the
natural inclusion $Spin^c(n-1)\subset
Spin^c(n)$).
\end{lemma}

\begin{equation*}
\begin{array}{ccc}
\eta^*(P) & \xrightarrow{\qquad} & P\\[5pt]
\downarrow & & \downarrow \\[5pt]
SOF(V) & \xrightarrow{\quad\eta\quad} & SOF(Z) \\[5pt]
\qquad\searrow & & \swarrow\qquad \\[5pt]
 & Z & \\
\end{array}
\end{equation*}

\begin{lemma}
Consider the $S^1$-equivariant spin$^c$ structure
$\eta^*(P)\rightarrow SOF(V) \rightarrow Z$. The quotient of each of
the three components by the left $S^1$ action gives rise to a
spin$^c$ structure on $Z/S^1$, called \emph{the quotient of the
given spin$^c$ structure}.
$$
\begin{CD} \overline{P}:=\eta^*(P)/S^1\\
@VVV \\
SOF(Z/S^1)=SOF(V)/S^1\\
@VVV \\
Z/S^1
\end{CD}
$$
\end{lemma}

\begin{remark}
If $L$ is the determinant line bundle of the given spin$^c$
structure on $Z$, then the determinant line bundle of $\overline{P}$
is $L/S^1$.
\end{remark}

\subsection{Spin$^c$ cutting}\label{spin-c-cutting}
We are now in the position of describing the process of cutting a
given $S^1$-equivariant spin$^c$ structure on a manifold. Assume
that the following data is given:
\begin{enumerate}
\item An oriented Riemannian manifold $M$ of dimension $m$.
\item An action of $S^1$ on $M$ by
isometries.
\item A co-oriented submanifold $Z\subset M$ of co-dimension 1 that is
$S^1$-invariant. We also demand that $S^1$ acts
freely on $Z$, and that $M\setminus Z$ is a
disjoint union of two open pieces $M_+$, $M_-$,
such that positive (resp.\ negative) normal
vectors point into $M_+$ (resp.\ $M_-$). Such
submanifolds are called \emph{reducible splitting
hypersurfaces} (see definitions 3.1 and 3.2 in
\cite{CKT}).
\item $P\rightarrow SOF(M)$ an $S^1$-equivariant spin$^c$ structure on $M$.
\end{enumerate}

We will use the following fact.
\begin{claim}
There is an invariant (smooth) function
$\Phi:M\to\mathbb{R}$, such that
$\Phi^{-1}(0)=Z$, $\Phi^{-1}(0,\infty)=M_+$,
$\Phi^{-1}(-\infty,0)=M_-$ and $0$ is a regular
value of~$\Phi$.
\end{claim}

To prove this claim, first define $\Phi$ locally
on a chart, use a partition of unity to get a
globally well defined function on the whole
manifold, and then average with respect to the
group action to get $S^1$-invariance.

 This function $\Phi$ plays the role of a
`moment map' for the $S^1$ action. To define
\emph{the cut space} $M_{cut}^+$, first introduce
an $S^1$-action on $M\times\mathbb{C}$
$$a\cdot(m,z)=(a\cdot m, a^{-1}z)$$ and then let
$M_{cut}^+=\left\{(m,z)|\Phi(m)=|z|^2\right\}/S^1$.
The cut space $M_{cut}^-$ is defined similarly,
using the diagonal action on $M\times\mathbb{C}$
$$a\cdot(m,z)=(a\cdot m, a\cdot z)$$ and by
setting
$M_{cut}^-=\left\{(m,z)|\Phi(m)=-|z|^2\right\}/S^1$.

\begin{remark}
The orientation and the Riemannian metric on $M$
(and on $\mathbb{C}$) descend to the cut spaces
$M_{cut}^\pm$ as follows. $M\times\mathbb{C}$ is
naturally an oriented Riemannian manifold.
Consider the map
$$\widetilde{\Phi}:M\times\mathbb{C}\to\mathbb{R}\qquad \widetilde{\Phi}(m,z)=\Phi(m)-|z|^2$$
Zero is a regular value of $\widetilde\Phi$, and
therefore $\widetilde{Z}=\widetilde\Phi^{-1}(0)$
is a manifold. It inherits a metric and is
co-oriented (hence oriented). Since $S^1$ acts
freely on $\widetilde Z$, the quotient
$M_{cut}^+=\widetilde Z/S^1$ is an oriented
Riemannian manifold (see Remark \ref{quotients}).

A similar procedure, using
$\widetilde{\Phi}(m,z)=\Phi(m)+|z|^2$, is carried
out in order to get an orientation and a metric
on $M_{cut}^-$.\\
We also have an $S^1$ action of the cut spaces
(see Remark \ref{M-action}).
\end{remark}

The purpose of this subsection is to describe how
to get spin$^c$ structures on $M_{cut}^\pm$ from
the given spin$^c$ structure on $M$. We start by
constructing a spin$^c$ structure on $M_{cut}^+$.

\begin{step}
Consider $\mathbb{C}$ with its natural structure
as an oriented Riemannian manifold, and let
$$P_\mathbb{C}=\mathbb{C}\times
Spin^c(2)\longrightarrow
SOF(\mathbb{C})=\mathbb{C}\times SO(2)
\longrightarrow\mathbb{C}$$ be the trivial
spin$^c$-structure on $\mathbb{C}$. Turn it into
an $S^1$-equivariant spin$^c$ structure by
letting $S^1$ act on $P_\mathbb{C}$:
$$e^{i\theta}\cdot(z,[a,b])=(e^{-i\theta}z,[x_{-\theta/2}\cdot a,e^{i\theta/2}\cdot b])
\qquad z\in\mathbb{C}\ ,\ [a,b]\in Spin^c(2) $$ where
$x_{\theta}=\cos\theta+\sin\theta\cdot e_1 e_2 \in
Spin(2)$.\\
Here is a diagram for this structure.

$$\begin{CD}
S^1\times P_{\mathbb{C}} @>>> P_\mathbb{C} @<<< P_\mathbb{C}\times Spin^c(2) \\
@VVV @VVV @VVV \\
S^1\times SOF(\mathbb{C}) @>>>
SOF(\mathbb{C}) @<<< SOF(\mathbb{C})\times SO(2) \\
@VVV @VVV \\
S^1\times\mathbb{C} @>>> \mathbb{C} \\
\end{CD}$$
\end{step}

\begin{step}
Taking the product of the spin$^c$ structures $P$
(on $M$) and $P_\mathbb{C}$ (on $\mathbb{C}$), we
get an ($S^1$ equivariant) spin$^c$ structure
$P_{M\times\mathbb{C}}$ on $M\times\mathbb{C}$
(see \S\ref{product}).
\end{step}

\begin{step}
It is easy to check that
$$ \widetilde{Z}=\{(m,z)|\Phi(m)=|z|^2\}\subset M\times\mathbb{C} $$
is an $S^1$-invariant co-oriented submanifold of
co-dimension one, and therefore we can restrict
$P_{M\times\mathbb{C}}$ and get an
$S^1$-equivariant spin$^c$ structure
$P_{\widetilde{Z}}$ on $\widetilde{Z}$ (see
\S\ref{rest}).
\end{step}

\begin{step}
Since $P_{\widetilde Z}\to SOF(\widetilde Z) \to
\widetilde Z$ is an $S^1$-equivariant spin$^c$
structure, we can take the quotient by the
$S^1$-action to get a spin$^c$ structure
$P_{cut}^+$ on $M_{cut}^+=\widetilde Z /S^1$ (see
\S\ref{quot-spin-c}).

\begin{remark} \label{M-action} The spin$^c$ structure $P_{cut}^+$
can be turned into an $S^1$-equivariant one. This
is done by observing that we actually have two
$S^1$ actions on $M\times\mathbb{C}$: the
anti-diagonal action $a\cdot(m,z)=(a\cdot m,
a^{-1}\cdot z)$ and the \emph{M-action}
$a\cdot(m,z)=(a\cdot m, z)$. These actions
commute with each other, and the M-action
naturally decends to the cut space $M_{cut}^+$
and lifts to the spin$^c$ structure $P_{cut}^+$.
\end{remark}
\end{step}

Let us now describe briefly the analogous
construction for $M_{cut}^-$.

\setcounter{step}{0}
\begin{step}
Define $P_\mathbb{C}$ as before, but with the action
$$e^{i\theta}\cdot(z,[a,b])=(e^{i\theta}z,[x_{\theta/2}\cdot a,e^{i\theta/2}\cdot b])
 $$
\end{step}

\begin{step}
Define the spin$^c$ structure
$P_{M\times\mathbb{C}}$ on $M\times\mathbb{C}$ as
before.
\end{step}

\begin{step}
As before, replacing $\widetilde{Z}$ with
$\{(m,z)|\Phi(m)=-|z|^2\}\subset M\times\mathbb{C}$ .
\end{step}

\begin{step}
Repeat as before to get a spin$^c$ structure
$P_{cut}^-$ on $M_{cut}^-$.
\end{step}

\begin{remark}\label{det_line_bndl}
In step 1 we defined a spin$^c$ structure on $\mathbb{C}$. The
corresponding determinant line bundle is the trivial line bundle
$\mathbb{L}_\mathbb{C}=\mathbb{C}\times\mathbb{C}$ over $\mathbb{C}$
(with projection $(z,b)\mapsto z$). The $S^1$ action on
$\mathbb{L}_\mathbb{C}$ is given by
\[
a\cdot(z,b)=
    \begin{cases}
    (a^{-1}\cdot z,a\cdot b) &\text{for $P_{cut}^+$}\\
    \\
    (a\cdot z, a\cdot b)     &\text{for $P_{cut}^-$}\\
    \end{cases}
\]
If $\mathbb{L}$ is the determinant line bundle of
the given spin$^c$ structure on $M$, then the
determinant line bundle on $M_{cut}^\pm$ is given
by

$$
\mathbb{L}_{cut}^\pm=\left[\left(\mathbb{L}\boxtimes\mathbb{L}_\mathbb{C}\right)|_
{\widetilde{Z}}\right]/S^1$$ where we divide by
the diagonal action of $S^1$ on
$\mathbb{L}\times\mathbb{L_C}$. This is an
$S^1$-equivariant complex line bundle (with
respect to the M-action).
\end{remark}

\section{The generalized Kostant formula for isolated fixed
points}\label{iso}

\noindent Assume that the following data is given:
\begin{enumerate}
\item An oriented compact Riemannian manifold $M$ of dimension
$2m$.
\item $T=\mathbb{T}^n$ an $n$-dimensional torus that acts on $M$ by
isometries.
\item $P\to SOF(M)$ a $T$-equivariant spin$^c$ structure, with
determinant line bundle $\mathbb{L}$.
\item A $U(1)$-invariant connection on $P_1=P/Spin(2m)$.
\end{enumerate}
As we saw in \S\ref{spin^c quantization}, this data determines a
complex virtual representation $Q(M)=ker(D^+)-coker(D^+)$ of
$T$.  Denote by $\chi\colon T\to\mathbb{C}$ its character.\\

\begin{lemma}\label{weight-lemma}
Let $x\in M^T$ be a fixed point, and choose a
$T$-invariant complex structure $J:T_xM\to T_xM$.
Denote by
$\alpha_1,\dots,\alpha_m\in\mathfrak{t}^*=Lie(T)^*$
the weights of the action $T\circlearrowright
T_xM$, and by $\mu$ the weight of
$T\circlearrowright\mathbb{L}_x$. Then
$\frac{1}{2}\left(\mu-\sum_{j=1}^m\alpha_j\right)$
is in the weight lattice of \ $T$.
\end{lemma}

\begin{proof}
Decompose $T_xM=L_1\oplus\cdots\oplus L_m$, where
each $L_j$ is a 1-dimensional $T$-invariant
complex subspace of $T_xM$, on which $T$ acts
with weight $\alpha_j$. Fix a point $p\in P_x$.

For each $z\in T$, there is a unique element
$[A_z,w_z]\in Spin^c(2m)$ such that $z\cdot
p=p\cdot [A_z,w_z]$. This gives a homomorphism
$$\eta\colon T\to Spin^c(2m)\qquad,\qquad
z\mapsto [A_z,w_z]$$ (note that $A_z$ and $w_z$
are defined only up to sign, but the element
$[A_z,w_z]$ is well defined).

Choose a basis $\{e_j\}\subset T_xM$ (over
$\mathbb{C}$) with $e_j\in L_j$ for all $1\le j
\le m$. With respect to this basis, each element
$z\in T$ acts on $T_xM$ through the matrix
\[ A'_z=\left(
\begin{array}{cccc}
z^{\alpha_1} &  &  & 0 \\
           & z^{\alpha_2} & & \\
           &            & \ddots & \\
  0        &  & & z^{\alpha_m}\\

\end{array}\right)\in U(m)\subset SO(2m)\ . \]
This enables us to define another homomorphism
$$\eta'\colon T\to SO(2m)\times S^1\qquad,\qquad
z\mapsto (A'_z,z^\mu)\ .$$

It is not hard to see that the relation $z\cdot
p=p\cdot [A_z,w_z]$ (for all $z\in T$) will imply
the commutativity of the following diagram.

\begin{diagram}
        &       &Spin^c(2m)      \\
        &\ruTo^{\eta}  &\dTo\\
T       &\rTo_{\eta'\quad}   &SO(2m)\times S^1      \\
\end{diagram}

\noindent(The vertical map is the double cover
taking $[A,z]\in Spin^c(2m)$ to
$(\lambda(A),z^2)$.\ ) For any $z=e^{i\theta}\in
T$ we have
$$\lambda(A_z)=A'_z\quad \Rightarrow\quad A_z=\prod_{j=1}^m\left[\cos\left(\frac{\theta\cdot\alpha_j}{2}\right)+
\sin\left(\frac{\theta\cdot\alpha_j}{2}\right)e_j
J(e_j)\right]\in Spin(2m)$$ (where the spin group
is thought of as sitting inside the Clifford
algebra)\\

and $$w_z^{\; 2}=z^\mu\qquad \Rightarrow \qquad
w_z=z^{\mu/2}\ .$$

Note that
$$T_{Spin^c(2m)}=\left\{\left[\prod_{j=1}^m\left(\cos t_j+
\sin t_j\cdot e_jJ(e_j)\right),u \right]\colon
t_j\in\mathbb{R}\ ,\ u\in S^1\right\}\subset
Spin^c(2m)$$ is a maximal torus, and that in fact
$\eta$ is a map from $T$ to $T_{Spin^c(2m)}$.

Now define another map $$\psi\colon
T_{Spin^c(2m)}\to S^1\qquad,\qquad
\left[\prod_{j=1}^m\left(\cos t_j+ \sin t_j\cdot
e_jJ(e_j),u\right) \right]\mapsto u\cdot
e^{-i\sum_j t_j}$$

By composing $\eta$ and $\psi$ we get a well
defined map $\psi\circ\eta\colon T\to S^1$ which
is given by
$$e^{i\theta}\mapsto
{\left(e^{i\theta}\right)}^{\frac{1}{2}\left(\mu-\sum_j
\alpha_j\right)}$$ and therefore
$\frac{1}{2}\left(\mu-\sum_j \alpha_j\right)$
must be a weight of $T$.
\end{proof}

\begin{remark}
The idea in the above proof is simple. To show
that $\beta=\frac{1}{2}\left(\mu-\sum_j
\alpha_j\right)$ is a weight, we want to
construct a 1-dimensional complex representation
of $T$ with weight $\beta$. The map $\eta$ is a
natural homomorphism $T\to Spin^c(2m)$. The map
$\psi$ is nothing but the action of a maximal
torus of $Spin^c(2m)$ on the lowest weight space
of the spin representation $\Delta^+_{2m}$ (see
Proposition \ref{T:spinrep}, and Lemma 12.12 in
\cite{GK}). Finally, $\psi\circ\eta\colon T\to
S^1$ is the required representation.
\end{remark}

The following is proposition 11.3 from \cite{GK}.
\begin{prop}\label{char. for isolated f.p.}
Assume that the fixed points $M^T$ of the action
on $M$ are isolated. For each $p\in M^T$, choose
a complex structure on $T_p M$, and denote by
\begin{enumerate}
\item $\alpha_{1,p},\dots,\alpha_{m,p}\in\mathfrak{t}^*$ the weights of the action of
$T$ on $T_p M$.
\item $\mu_p$ the weight of the action of $T$ on $\mathbb{L}_p$.
\item $(-1)^p$ will be $+1$ if the orientation coming from the
choice of the complex structure on $T_p M$ coincides with the
orientation of $M$, and $-1$ otherwise.
\end{enumerate}
Then the character $\chi\colon T\to\mathbb{C}$ of
$Q(M)$ is given by
$$\chi(\lambda)=\sum_{p\in M^G}\nu_p(\lambda)\qquad\quad
\nu_p(\lambda)=(-1)^p\cdot\lambda^{\mu_p
/2}\prod_{j=1}^m
\frac{\lambda^{-\alpha_{j,p}/2}-\lambda^{\alpha_{j,p}/2}}
{(1-\lambda^{\alpha_{j,p}})(1-\lambda^{-\alpha_{j,p}})}$$
where $\lambda^\beta:T\to S^1$ is the
representation that corresponds to the weight
$\beta\in\mathfrak{t}^*$.

\begin{remark} \
\begin{enumerate}
\item Although  $\pm\alpha_{j,p}/2$ may not be in the weight
lattice of $T$, the expression $\nu_p(\lambda)$,
can be equivalently written as
$$(-1)^p\cdot\lambda^{\left(\mu_p-\sum_j\alpha_{j,p}\right)/2}\prod_{j=1}^m
\frac{1-\lambda^{\alpha_{j,p}}}
{(1-\lambda^{\alpha_{j,p}})(1-\lambda^{-\alpha_{j,p}})}\
.$$ By Lemma \ref{weight-lemma},
$\left(\mu_p-\sum_j\alpha_{j,p}\right)/2$ is a
weight, so $\nu_p(\lambda)$ is well defined.
\item Since the fixed points of the action $T\circlearrowright M$
are isolated, all the $\alpha_{j,p}$'s are nonzero. This follows
easily from theorem B.26 in \cite{Kar}.
\end{enumerate}
\end{remark}
\end{prop}

Now we present the generalized Kostant formula for spin$^c$
quantization.\\
Assume that the fixed points of
$T\circlearrowright M$ are isolated, choose a
complex structure on $T_p M$ for each $p\in M^G$,
and use the notation of Proposition \ref{char.
for isolated f.p.}. By the above remark, we can
find a \emph{polarizing vector}
$\xi\in\mathfrak{t}$ such that
$\alpha_{j,p}(\xi)\neq 0$ for all $j,p$. We can
choose our complex structures on $T_p M$ such
that $\alpha_{j,p}(\xi)\in i\mathbb{R}^+$ for all
$j,p$.

For each weight $\beta\in\mathfrak{t}^*$ denote
by $\#(\beta,Q(M))$ the multiplicity of this
weight in $Q(M)$. Also, for $p\in M^T$ define the
partition function
$\overline{N}_p:\mathfrak{t}^*\to\mathbb{Z}^+$ by
setting:
$$\overline{N}_p(\beta)=\left|\left\{(k_1,\dots,k_m)\in\left(\mathbb{Z}+\frac{1}{2}\right)^m\;:\;\beta+\sum_{j=1}^m
k_j \alpha_{j,p}=0\quad,\quad k_j>0 \right\}\right|$$ The right hand
side is always finite since our weights are polarized.

\begin{theorem}[Kostant formula] \label{Kost}
For any weight $\beta\in\mathfrak{t}^*$ of $T$,
we have
$$\#(\beta,Q(M))=\sum_{p\in
M^G}(-1)^p\cdot\overline{N}_p\left(\beta-\frac{1}{2}\mu_p\right)$$
\end{theorem}

\begin{proof}
For $p\in M^T$ and $\lambda\in T$, set
$\alpha_j=\alpha_{j,p}$ and $\mu=\mu_p$. From
Proposition \ref{char. for isolated f.p.} we then
get
\begin{multline*}
\nu_p(\lambda)
=(-1)^p\cdot\lambda^{\mu/2}\prod_{j=1}^m\frac{\lambda^{-\alpha_j/2}(1-\lambda^{\alpha_j})}
{(1-\lambda^{\alpha_j})(1-\lambda^{-\alpha_j})}=
(-1)^p\cdot\lambda^{\frac{1}{2}(\mu-\sum_j{\alpha_j})}
\prod_{j=1}^m\frac{1}{1-\lambda^{-\alpha_j}}
\end{multline*}

\noindent Note that we have
$$ \prod_{j=1}^m\frac{1}{1-\lambda^{-\alpha_j}}=\sum_\beta N_p(\beta)\cdot\lambda^\beta$$
Where the sum is taken over all weights
$\beta\in\mathfrak{t}^*$ in the weight lattice
$\ell^*$ of $T$ and $N_p(\beta)$ is the number of
\emph{non-negative} integer solutions
$(k_1,\dots,k_m)\in(\mathbb{Z}_+)^m$ to
$$\beta+\sum_{j=1}^m k_j\alpha_j=0$$
(see formula 5 in \cite{GSW}). Hence,
$$\nu_p(\lambda)=(-1)^p\cdot\sum_{\beta\in\ell^*}N_p(\beta)\cdot\lambda^{\beta+\frac{1}{2}(\mu-\sum_j{\alpha_j})}$$
By Lemma \ref{weight-lemma},
$\frac{1}{2}(\mu-\sum_j{\alpha_j})\in\ell^*$
(i.e., it is a weight), so by change of variable
$\beta\mapsto\beta-\frac{1}{2}(\mu-\sum_j{\alpha_j})$
we get
$$\nu_p(\lambda)=(-1)^p\cdot\sum_{\beta\in\ell^*}N_p\left(\beta-\frac{1}{2}\mu+\frac{1}{2}\sum_j{\alpha_j}\right)\cdot\lambda^{\beta}$$
By definition,
$N_p\left(\beta-\frac{1}{2}\mu+\frac{1}{2}\sum_j{\alpha_j}\right)$
is the number of non-negative  integer solutions for the equation
$$\beta-\frac{1}{2}\mu+\frac{1}{2}\sum_j{\alpha_j}+\sum_{j} k_j\alpha_j=0$$
or, equivalently, to
$$\beta-\frac{1}{2}\mu+\sum_{j} \left(k_j+\frac{1}{2}\right)\alpha_j=0$$
Using the definition of $\overline N_p$ (see above) we conclude that
$$N_p\left(\beta-\frac{1}{2}\mu+\frac{1}{2}\sum_j{\alpha_j}\right)=\overline
N_p\left(\beta-\frac{1}{2}\mu\right)$$ and then
$$
\nu_p(\lambda)=(-1)^p\cdot\sum_{\beta\in\ell^*}\overline{N}_p\left(\beta-\frac{1}{2}\mu\right)\lambda^\beta$$
This means that the formula to the character can be written as
$$\chi(\lambda)=\sum_{\beta\in\ell^*}\left[\sum_{p\in M^G} (-1)^p\cdot\overline{N}_p\left(\beta-\frac{1}{2}\mu\right)  \right]
\lambda^\beta$$ and the multiplicity of $\beta$ in $Q(M)$ is given
by $$\#(\beta,Q(M))=\sum_{p\in M^G}
(-1)^p\cdot\overline{N}_p\left(\beta-\frac{1}{2}\mu\right)$$ as
desired.
\end{proof}

\section{The generalized Kostant formula for non-isolated fixed
points}\label{noniso-sec}

\subsection{Equivariant characteristic classes} \ \\
Let an \emph{abelian} Lie group $G$ (with Lie
algrbra $\mathfrak{g}$) act \emph{trivially} on a
smooth manifold $X$. We now define the
equivariant cohomology (with generalized
coefficients) and equivariant characteristic
classes for this special case. For the more
general case, see  \cite{KV} or Appendix C in
\cite{Kar}.

\begin{definition}
A real-valued function $\alpha$ is called an
\emph{almost everywhere analytic function}
(a.e.a) if
\begin{enumerate}
\item Its domain is of the form $\mathfrak{g}\setminus P$, and $P\subset\mathfrak{g}$ is a closed set of measure zero.
\item It is analytic on $\mathfrak{g}\setminus P$.
\end{enumerate}
Denote by $C^\#(\mathfrak{g})$ the space of all
equivalence classes of a.e.a functions on
$\mathfrak{g}$ (two such functions are equivalent
if they coincide outside a set of measure zero).
\end{definition}

Let
$\mathcal{A}_G^\#(X)=C^\#(\mathfrak{g})\otimes\Omega^\bullet(X;\mathbb{C})$
be the space of all a.e.a functions
$\mathfrak{g}\to\Omega^\bullet(X;\mathbb{C})$,
where $\Omega^\bullet(X;\mathbb{C})$ is the
(ordinary) de Rham complex of $X$ with complex
coefficients.

Define a differential (recall that $G$ is abelian
and the action is trivial)
$$
d_\mathfrak{g}:\mathcal{A}_G^{\#}(X)\to\mathcal{A}_G^{\#}(X)\qquad\qquad
(d_\mathfrak{g}\alpha)(u)=d(\alpha(u))$$ and the \emph{G-equivariant
(de Rham) cohomology} of X
$$H_G^\#(X)=\frac{Ker(d_\mathfrak{g})}{Im(d_\mathfrak{g})}\ .$$
Note that $H_G^\#(X)$ is isomorphic to the space $C^\#(\mathfrak{g})\otimes H^\bullet(X;\mathbb{C})$ of a.e.a functions $\mathfrak{g}\to H^\bullet(X;\mathbb{C})$.
Equivariant
characteristic classes will be elements of the ring $H_G^\#(X)$.

If $X$ is compact and oriented, then equivariant cohomology classes
can be integrated over $X$. For any class $[\alpha]\in
H_G^\#(X)$ and $u$ in the domain of $\alpha$, let
$$\left(\int_X [\alpha]\right)(u)=\int_X (\alpha(u))$$
and thus $\int_X [\alpha]$ is an element of
$C^\#(\mathfrak{g})\otimes\mathbb{C}$.

Assume now that both $X$ and $G$ are connected,
and let $\pi\colon L\to X$ be a complex line
bundle over $X$. Assume that $G$ acts on the
fibers of the bundle with weight $\mu\in
\mathfrak{g}^*$, i.e., $\exp(u)\cdot y=
e^{i\mu(u)}\cdot y$ for all $u\in
\mathfrak{g}\;\text{ and }\;y\in L$ (so the
action on the base space is still trivial).
Denote by $c_1(L)=[\omega]\in H^2(X)$ the
(ordinary) first Chern class of the line bundle.
Here $\omega\in\Omega^2(X)$ is a real two-form.
Then the \emph{first equivariant Chern class} of
the equivariant line bundle $L\to X$ is defined
to be $[\omega+\mu]\in H^\#_G(X)$. We will denote
this class by $\tilde{c_1}(L)$.

Now assume that $E\to X$ is a $G$-equivariant complex vector bundle
of complex rank $k$ (where $G$ acts trivially on $X$), that splits
as a sum of $k$ equivariant complex line bundles
$E=L_1\oplus\cdots\oplus L_k$ (one can avoid this assumption by
using \emph{the (equivariant) splitting principle}). Let
$\tilde{c_1}(L_1)=[\omega_1+\mu_1],\cdots,\tilde{c_1}(L_k)=[\omega_k+\mu_k]$ be the equivariant first
Chern classes of these line bundles, and define \emph{the
equivariant Euler class} of $E$ by

$$ \tilde{Eu}(E)=\prod_{j=1}^k
\tilde{c_1}(L_j)=\left[\prod_{j=1}^k (\omega_j+\mu_j)\right]
\in H_{G}^\#(X) .$$

We will also need the equivariant $A$-roof  class, which we will denote
by $\tilde{A}(E)$. To define this class, consider the following
meromorphic function
$$f(z)=\frac{z}{e^{z/2}-e^{-z/2}}=\frac{z/2}{\sinh(z/2)}\qquad\qquad f(0)=1\ .$$
Its domain is $D=\mathbb{C}\setminus\{\pm 2\pi i, \pm 4\pi i, \dots \}$.
Define, for each $1\le j\le k$,

$$f(\tilde{c_1}(L_j))(u)=f(c_1(L_j)+\mu_j(u))=\sum_{n=1}^\infty\frac{f^{(n)}(\mu_j(u))}{n!}\cdot
\left(c_1(L_j)\right)^n$$
whenever $\mu_j(u)\in D$ for all $1\le j \le k$, and also
$$ \tilde{A}(E)=\prod_{j=1}^k f(\tilde{c_1}(L_j)).$$

Also note that the quotient $$\frac{\tilde{A}(E)}{\tilde{Eu}(E)}$$ can be defined using the same procedure, replacing $f(z)$ with \ $\frac{1}{2\sinh(z/2)}$ .
If all the $\mu_j$'s are nonzero, then $$\frac{\tilde{A}(E)}{\tilde{Eu}(E)}\in H^\#_G(X) .$$

\subsection{The Kostant Formula}\label{noniso} \ \\
\noindent Assume that the following data is given:
\begin{enumerate}
\item An oriented compact Riemannian manifold $M$ of dimension
$2m$.
\item A circle action $S^1\circlearrowright M$ by isometries.
\item  An $S^1$-equivariant spin$^c$ structure $P\to SOF(M)$, with
determinant line bundle $\mathbb{L}$.
\item A $U(1)$-invariant connection on $P_1=P/Spin(2m)\to M$.
\end{enumerate}

In this section we present a formula for the
character $\chi\colon S^1\to\mathbb{C}$ of the
virtual representation $Q(M)$ determined by the
above data (see \S\ref{spin^c quantization}). We
\emph{do not} assume, however, that the fixed
points are isolated.

We use the following conventions and notation.
\begin{itemize}
\item $M^{S^1}$ is the fixed points set.
\item For each connected component $F\subset M^{S^1}$,
let $NF$ denote the normal bundle to $TF\subset
TM$. The bundles $NF$ and  $TF$ are
$S^1$-equivariant real vector bundles of even
rank, with trivial fixed subspace, and therefore
are equivariantly isomorphic to complex vector
bundles. Choose an equivariant  complex structure
on the fibers of $TF$ and $NF$, and denote the
rank of $NF$ as a complex vector bundle by
$m(F)$.
\item The complex structures on $NF$ and $TF$ induce an orientation on those bundles. Let $(-1)^F$ be $+1$ if the orientation of $F$ followed by that of $NF$ is the given orientation on $M$, and $-1$ otherwise.
\end{itemize}
With respect to the above data, choices and notation, we have

\begin{prop}
For all $u\in\mathfrak{g}=Lie(S^1)$ such that the right hand side is defined,
$$\chi(exp(u))=\sum_{F\subset M^{S^1}} (-1)^F\cdot (-1)^{m(F)}\cdot\int_F
e^{\frac{1}{2}\tilde{c_1}(\mathbb{L}|_F)}\cdot\tilde{A}(TF)\cdot\frac{\tilde{A}(NF)}{\tilde{Eu}(NF)}$$
where the sum is taken over the connected components of $M^{S^1}$.
\end{prop}

This formula is derived from the Atiyah-Segal-Singer index theorem (see \cite{ASIII}). For some details, see p.547 in \cite{CKT}.

 Assume that the normal bundle splits as a
direct sum of (equivariant) complex line bundles
$$NF=L_1^F\oplus\cdots L_{m(F)}^F\quad .$$

For each fixed component $F\subset M^{S^1}$,
denote by $\{\alpha_{j,F}\}$ the weights of the
action of $S^1$ on $\{L_j^F\}$. As in the
previous section, all the $\alpha_{j,F}$'s are
nonzero, and we can polarize them, i.e., we can
choose our complex structure on $NF$ in such a
way that $\alpha_{j,F}(\xi)>0$ for some fixed
$\xi\in\mathfrak{g}$ and for all $j$'s and $F$'s.
Also denote by $\mu_F$ the weight of the action
of $S^1$ on $\mathbb{L}|_F$.

For each $\beta\in\mathfrak{g}^*=Lie(S^1)^*$, define the following set (which is finite, since our weights are polarized)
$$ \mathcal{S}_\beta=\left\{ (k_1,\dots,k_{m(F)})\in\left(\mathbb{Z}+\frac{1}{2}\right)^{m(F)}\quad:\quad \beta+\sum_{j=1}^{m(F)}k_j\alpha_{j,F}=0\quad,\quad k_j>0\right\}$$

and for each tuple $k=(k_1,\dots,k_{m(F)})$, let $$\overline p_{k,F}=(-1)^{m(F)}\int_F e^{\frac{1}{2}\left(c_1(\mathbb{L}|_F)-\sum_j c_1(L^F_j)\right)}\cdot\tilde{A}(TF)\cdot e^{-\sum_j k_jc_1(L^F_j)}\quad .$$
Now define $$\overline{N}_F(\beta)=\sum_{k\in\mathcal{S}_\beta}\overline{p}_{k,F}\quad .$$

With this notation, the Kostant formula in this
case of nonisolated fixed points becomes
identical to the formula for isolated fixed
points (from \S\ref{iso}).

\begin{theorem}\label{Kos-noniso}
For each weight
$\beta\in\mathfrak{g}^*=Lie(S^1)^*$, the
multiplicity of $\beta$ in $Q(M)$ is given by
$$\#(\beta,Q(M))=\sum_{F\subset
M^{S^1}}(-1)^F\cdot\overline{N}_F\left(\beta-\frac{1}{2}\mu_F\right)\
,$$ where the sum is taken over the connected
components of $M^{S^1}$.
\end{theorem}

\begin{proof}
For a fixed connected component $F\subset M^{S^1}$, omit the  $F$ in $\alpha_{j,F}$,  $\mu_F$ and $L_j^F$, and compute
\begin{multline*}
\int_F e^{\frac{1}{2}\tilde{c}_1\left(\mathbb{L}|_F\right)}\cdot\tilde{A}(TF)\cdot\frac{\tilde{A}(NF)}{\tilde{Eu}(NF)}=\hfill\\[10pt]
\indent\indent=\int_F
e^{\frac{1}{2}c_1(\mathbb{L}|_F)+\frac{1}{2}\mu}\cdot\tilde{A}(TF)\cdot
\prod_{j=1}^{m(F)}\frac{1}{e^{[c_1(L_j)+\alpha_j]/2}-e^{-[c_1(L_j)+\alpha_j]/2}}=\hfill\\[10pt]
\indent\indent=e^{\frac{1}{2}\mu}\cdot\int_F
e^{\frac{1}{2}c_1(\mathbb{L}|_F)}\cdot
\tilde{A}(TF)\cdot\displaystyle\prod_{j=1}^{m(F)}\frac{e^{-[c_1(L_j)+\alpha_j]/2}}{1-e^{-[c_1(L_j)+\alpha_j]}}=\hfill\\[10pt]
\indent\indent=e^{\left[\mu-\sum_j
\alpha_j\right]/2}\cdot\int_F
e^{\left[c_1(\mathbb{L}|_F)-\sum_j
c_1(L_j)\right]/2}\cdot
\tilde{A}(TF)\cdot\prod_{j=1}^{m(F)}\frac{1}{1-e^{-\left[ c_1(L_j)+\alpha_j\right]}}\hfill\\
\
\end{multline*}
Using the geometric series
$$\frac{1}{1-z}=\sum_{l=0}^\infty z^l$$
and the notation $z=\exp(u)$ we get, for each
$j$, and for each $u\in\mathfrak{g}$ such that
the series converges,
$$\frac{1}{1-e^{-\left[c_1(L_j)+\alpha_j(u)\right]}}=\sum_{l=0}^\infty
e^{-l\cdot\left[c_1(L_j)+\alpha_j(u)\right]}=\sum_{l=0}^\infty
e^{-l\cdot c_1(L_j)}z^{-l\cdot\alpha_j}$$ (where
$z^{-l\cdot\alpha_j}$ is the representation of $S^1$ that
corresponds to the weight
$-l\cdot\alpha_j\in\ell^*\subset\mathfrak{g}^*$) and thus
$$\prod_{j=1}^{m(F)} \frac{1}{1-e^{-\left[c_1(L_j)+\alpha_j(u)\right]}}=
\sum_{l\in\ell^*}\left[\sum_{l+\sum_j
k_j\alpha_j=0}e^{-\sum_j k_j
c_1(L_j)}\right]z^l\vspace{10pt}$$ The formula
that we get for the character is\vspace{10pt}
\begin{multline*}
\chi(z)=\sum_{F\subset M^{S^1}}(-1)^F\cdot (-1)^{m(F)}\cdot e^{\frac{1}{2}\left[\mu_F(u)-\sum_j\alpha_{j,F}(u)\right]}\cdot\\
\indent\indent\cdot\int_F
e^{\frac{1}{2}\left[c_1(\mathbb{L}|_F)-\sum_j
c_1(L_j^F)\right]}\cdot\tilde{A}(TF)\cdot
\sum_{l\in\ell^*}\left[\sum_{l+\sum_j k_j\alpha_{j,F}=0}e^{-\sum_j k_j c_1(L_{j,F})}\right]z^l=\hfill\\[10pt]
\indent\indent=\sum_{l\in\ell^*}\ \sum_{F\subset
M^{S^1}}\  \sum_{l+\sum_j
k_j\alpha_{j,F}=0}(-1)^F\cdot
\overline{p}_{k,F}\cdot
z^{l+\frac{1}{2}\left(\mu_F-
\sum_j\alpha_{j,F}\right)}\hfill
\end{multline*}
Lemma \ref{weight-lemma} implies that
$\frac{1}{2}\left(\mu_F-
\sum_j\alpha_{j,F}\right)$ is a weight of $S^1$
(so the previous formula is well defined), hence
we can make a change of variables
$\beta=l+\frac{1}{2}\mu_F-
\frac{1}{2}\sum_j\alpha_{j,F}$ and get
$$\chi(z)=\sum_{\beta\in\ell^*}\left[\sum_{F\subset M^{S^1}}
\sum_{k\in\mathcal
S_{\beta-\frac{1}{2}\mu_F}}(-1)^F\cdot
\overline{p}_{k,F}\right] z^\beta=
\sum_{\beta\in\ell^*}\left[\sum_{F\subset
M^{S^1}}(-1)^F\cdot\overline
N_F\left(\beta-\frac{1}{2}\mu_F\right)\right]z^\beta
\vspace{10pt}$$ From this we conclude that the
multiplicity of
$\beta\in\ell^*\subset\mathfrak{g}^*$ in $Q(M)$
is given by
$$\#(\beta,Q(M))=\sum_{F\subset
M^{S^1}}(-1)^F\cdot\overline
N_F\left(\beta-\frac{1}{2}\mu_F\right)$$ as
desired (the sum is taken over the connected
components of the fixed point set $M^{S^1}$).
\end{proof}

\subsection{The case $m(F)=1$} \

To prove the additivity of spin$^c$ quantization
under cutting, we will need the terms of the
Kostant formula for non-isolated fixed points in
the special case where $m(F)=1$, i.e., when the
normal bundle to the fixed components has complex
dimension $1$. Therefore, assume that we are
given the same data as in \S\ref{noniso}, and
also that

\begin{itemize}
\item Each fixed component $F\subset M^{S^1}$ is of real codimension $2$ in $M$, i.e., the normal bundle $NF=TM/TF$ is of real dimension $2$.
\end{itemize}

For a fixed component $F$, we adopt all the notation from \S\ref{noniso}. Since $m(F)$ is assumed to be $1$, we have $$NF=L^F_1$$ and only one weight $$\alpha_{1,F}=\alpha_F .$$
For each $\beta\in\mathfrak{g}^*$, the corresponding set $\mathcal{S}_\beta$ becomes
$$\mathcal{S}_\beta=\left\{ k\in \mathbb{Z}+\frac{1}{2}\quad\colon\quad
\beta+k\cdot\alpha_F=0\quad,\quad k>0\right\}$$
which is either empty or contains only one element.
The expression for $\overline p_{k,F}$ also simplifies to
$$\overline p_{k,F}=-\int_F e^{\left[ c_1(\mathbb{L}|_F)-c_1(NF)\right]/2}\cdot\tilde
A(TF)\cdot e^{-k\cdot c_1(NF)}\ ,$$ and this
implies that
\begin{equation*}
\overline N_F\left(\beta-\frac{1}{2}\mu_F\right)=
\begin{cases}
0 & \mbox{if \ $\mathcal{S}_{\beta-\frac{1}{2}\mu_F}=\phi$}\\
\overline p_{k,F} &\mbox{if \ $\mathcal{S}_{\beta-\frac{1}{2}\mu_F}=\{k\}$}
\end{cases} \ .
\end{equation*}

\section{Additivity under
cutting}\label{additivity} In this section we
prove our main result, namely, the additivity of
spin$^c$ quantization under the cutting
construction described in \S\ref{spin-c-cutting}
.

Our setting is as follows:
\begin{enumerate}
\item A compact oriented connected Riemannian manifold $M$ of dimension~$2m$.
\item An action of $S^1$ on $M$ by isometries.
\item An $S^1$-equivariant spin$^c$ structure $P\to SOF(M)\to M$.
\item A co-oriented splitting hypersurface $Z\subset M$ on which $S^1$ acts freely.
\end{enumerate}

After choosing a $U(1)$-invariant connection on $P_1=P/Spin(2m)$, we can construct a Dirac operator $D^+$, whose index $Q(M)$ is independent of the connection. We call $Q(M)$ \emph{the spin$^c$ quantization of $M$} (see \S\ref{spin^c quantization}).

We can now perform the cutting construction from \S\ref{spin-c-cutting} to obtain two other manifolds $M_{cut}^\pm$ (the \emph{cut spaces}). Those cut spaces are also compact oriented Riemannian manifolds of dimension $2m$, endowed with a circle action and with $S^1$-equivariant spin$^c$ structures $P_{cut}^\pm$. Thus, we can quantize them (after choosing a suitable connection), and obtain two virtual representations $Q\left(M_{cut}^\pm\right)$.

\begin{theorem}\label{mainthm}
As virtual representations of $S^1$, we have
$$Q(M)=Q\left(M_{cut}^+\right)\oplus Q\left(M_{cut}^-\right)$$
\end{theorem}

We will need a few preliminary lemmas for the
proof of the theorem. Those are similar to
Proposition 6.1 from \cite{CKT}, where a few gaps
where found.

\subsection{First lemma - the normal bundle} \ \\
Recall the construction of $M_{cut}^\pm$ from section
\ref{spin-c-cutting}.
\begin{itemize}
\item Choose an $S^1$-invariant smooth function $\phi\colon
M\to\mathbb{R}$ such that $\phi^{-1}(0)=Z$,\
$\phi^{-1}(0,\infty)=M_{+}$, \
$\phi^{-1}(-\infty,0)=M_-$, and $0$ is a regular
value of $\phi$.
\item Define $\tilde Z^{\pm} =\left\{(m,z)\;|\;\phi(m)=\pm|z|^2\right\}\subset
M\times\mathbb{C}$,
and let $S^1$ act on $\tilde Z^{\pm}$ by $a\cdot(m,z)=(a\cdot m,
a^{\mp 1}\cdot z)$.
\item Finally, define $M_{cut}^\pm=\tilde Z^\pm / S^1$ .
\end{itemize}

\begin{remark}
Note that we have $S^1$-equivariant embeddings
$$ Z\to \tilde Z^\pm \ , \  m\mapsto
(m,0)\qquad\quad\mbox{and}\qquad\quad
 Z/S^1 \to M_{cut}^\pm\ , \  [m]\mapsto [m,0]$$
and therefore we can think of $Z$ and $Z/S^1$ as
submanifolds of $\tilde Z^\pm$ and $M_{cut}^\pm$,
respectively.
\end{remark}

\begin{lemma} \
\begin{enumerate}\label{norm_lemma}
\item The maps $$\eta\colon T(\tilde Z^\pm)|_Z\to Z\times\mathbb{C}\qquad\eta\colon (v,w)
\in T_{(m,0)}\tilde Z^\pm\mapsto (m,w)$$ give rise to short exact
sequences
$$ 0\longrightarrow TZ\longrightarrow T\tilde Z^\pm|_Z \xrightarrow{\ \eta\ } Z\times\mathbb{C}\longrightarrow 0$$
of $S^1$-equivariant vector bundles (with respect
to both the diagonal (anti-diagonal) action and
the M-action) over $Z$. The action on
$Z\times\mathbb{C}$ is taken to be
$$a\cdot(m,z)=(a\cdot m, a^{\mp 1}\cdot z) \ .$$
\item The short exact sequences above descend to  the following short exact
sequences
$$ 0\longrightarrow T(Z/S^1)\longrightarrow T(M_{cut}^\pm)|_{Z/S^1}
\longrightarrow
Z\times_{S^1}\mathbb{C}\longrightarrow 0$$ of
equivariant vector bundles over $Z/S^1$. The
$S^1$ action on $Z\times_{S^1}\mathbb{C}$ is
induced from the action on $Z$.
\end{enumerate}
\end{lemma}
\begin{proof} \
\begin{enumerate}
\item The $S^1$-equivariant embedding $Z\to
\widetilde Z^\pm$ gives rise to an injective map
$TZ\to T\widetilde Z^\pm$, which is an
$S^1$-equivariant map of vector bundles over $Z$.
The map $\eta$ is onto, since for any $(m,w)\in
Z\times\mathbb{C}$ we have $\eta(0,w)=(m,w)$, and
it is equivariant since for $(v,w)\in
T_{(m,0)}\widetilde Z^\pm\ , \ m\in Z$ we have
$$\eta(a\cdot (v,w))=\eta(a\cdot v,a^\mp\cdot
w)=(a\cdot m,a^\mp\cdot w)=a\cdot (m,z)$$ (and
similarly for the M-action).

To prove $ker(\eta)=TZ$, note that the
definitions of $\phi$ and $\widetilde Z$ imply
that
\begin{multline*}
\qquad\qquad T\widetilde Z^\pm=\left\{(v,w)\in
T_{(m,z)}M\times\mathbb{C}\;:\;
d\phi_m(v)=z\cdot\overline{w}+\overline{z}\cdot
w\right\}\\ \qquad\qquad\quad TZ =\left\{v\in
T_mM\;:\; d\phi_m(v)=0\right\}\hfill
\end{multline*}
so $(v,w)\in T_{(m,0)}\widetilde Z^\pm$ satisfies
$\eta(v,w)=(m,0)$ if and only if $$w=0\
\Leftrightarrow \ d\phi_m(v)=0\ \Leftrightarrow \
v=(v,0)\in T_m Z\subset T_{(m,0)}\widetilde
Z^\pm$$ and hence
$ker(\eta)=TZ$ and the sequence is exact.\\

\item is a direct consequence of (1).
\end{enumerate}
\end{proof}

Let $N^\pm\to Z$ be the normal bundle to $Z$ in
$\tilde Z^\pm$, and $\overline N^\pm\to Z/S^1$ be
the normal bundle to $Z/S^1$ in $M_{cut}^\pm$.
The above lemma implies:
\begin{cor}\label{Normal bundle}
The short exact sequences of Lemma
\ref{norm_lemma} induce isomorphisms
$$N^\pm\xrightarrow{\quad\simeq\quad} Z\times\mathbb{C}\qquad
\overline
N^\pm\xrightarrow{\quad\simeq\quad}Z\times_{S^1}\mathbb{C}$$
of equivariant vector bundle, and hence an
orientation on the fibers of the bundles
$\overline N^\pm$ (coming from the complex
orientation on $\mathbb{C}$).
\end{cor}

\begin{remark}
Note that the map
$$\overline{N}^+=Z\times_{S^1}\mathbb{C}\xrightarrow{\qquad}\overline{N}^{\;-}=Z\times_{S^1}\mathbb{C}\qquad,\qquad
[z,a]\mapsto[z,\overline a]$$ is an
$S^1$-equivariant orientation-reversing bundle
isomorphism.

\end{remark}

\begin{claim}\label{orien - normal}
The natural orientation on $Z/S^1\subset
M_{cut}^\pm$, coming from the reduction process,
followed by the orientation of $\overline N^\pm$,
gives the orientation on $M_{cut}^\pm$.
\end{claim}
\begin{proof}
Fix $x\in Z$. Choose an oriented orthonormal
basis for $T_xM$ of the form
$$v_1,\dots,v_{2m-2},v_\theta,v_N$$
where $v_\theta=c\cdot
\left(\frac{\partial}{\partial\theta}\right)_{M,x}$
is a positive multiple of the generating vector
field at $x\in Z$ ($c>0$ is chosen such that
$v_\theta$ has length 1),
$\{v_1,\dots,v_{2m-2},v_\theta\}$ are an oriented
orthonormal basis for $T_x Z$, and $v_N$ is a
positively oriented normal vector to $Z$.

By the definition of the metric and orientation
on the reduced space, the push-forward of
$v_1,\dots,v_{2m-2}$ by the quotient map $Z\to
Z/S^1$ is an oriented orhonormal basis for
$T_{[x]}\left(Z/S^1\right)$.

Now the vectors
$$v_1,\dots,v_{2m-2},1,i,v_\theta,v_N\in T_{(m,0)}M\times\mathbb{C}$$
are an oriented orthonormal basis, where
$1,i\in\mathbb{C}$. Note that
$\left(\frac{\partial}{\partial\theta}\right)_M=\left(\frac{\partial}{\partial\theta}\right)_{M\times\mathbb{C}}$
on $Z\cong Z\times\{0\}\subset
M\times\mathbb{C}$, and that the normal to $Z$ in
$M$ can be identified with the normal to $\tilde
Z^\pm$ in $M\times\mathbb{C}$, when restricted to
$Z\subset\tilde Z^\pm$. Hence, the push forward
of $v_1,\dots,v_{2m-2},1,i$ by the quotient map
$\tilde Z^\pm\to M_{cut}^\pm$ is an orthonormal
basis for $T_{[m,0]}M_{cut}^\pm$.

Since $1,i$ descend to an oriented orthonormal
basis for $(\overline N^\pm)_x$, when identified
with $\mathbb{C}$ using Corollary \ref{Normal
bundle}, the claim follows.
\end{proof}

\subsection{Second lemma - the determinant line bundle.} \ \\
We would like to relate the determinant line
bundles of $P_{cut}^\pm$ (over $M_{cut}^\pm$),
which will be denoted by $\mathbb{L}_{cut}^\pm$,
to the determinant line bundle $\mathbb{L}$ of
the spin$^c$ structure $P$ on $M$. Denote
$\mathbb{L}_{red}=\left(\mathbb{L}|_Z\right)/S^1$.
This is a line bundle over $Z/S^1\subset
M_{cut}^\pm$.

Then we have:

\begin{lemma}\ \label{det-reduced}
The restriction of $\mathbb{L}_{cut}^{\pm}$ to
$Z/S^1$ is isomorphic, as an $S^1$-equivariant
complex line bundle, to
$\mathbb{L}_{red}\otimes\overline N^{\;-}$.
\end{lemma}
\begin{remark}
This is not a typo. Both $\mathbb{L}^+_{cut}$ and
$\mathbb{L}^-_{cut}$ are isomorphic to
$\mathbb{L}_{red}\otimes\overline N^{\;-}$.
\end{remark}

\begin{proof}
Recall that the determinant line bundle over the
cut spaces is given by
$$\mathbb{L}_{cut}^\pm=\left[(\mathbb{L}\boxtimes\mathbb{L_C}^\pm)|_{\widetilde{Z}^\pm}\right]/S^1$$
where $\mathbb{L_C}^\pm$ is the determinant line
bundle of the spin$^c$ structure on $\mathbb{C}$,
defined in the process of constructing
$P_{cut}^\pm$, and we divide by the diagonal
action of $S^1$ on
$\mathbb{L}\times\mathbb{L_C}^\pm$.

Therefore we have
$$\mathbb{L}_{cut}^\pm|_{Z/S^1}=\left[(\mathbb{L}\boxtimes\mathbb{L_C}^\pm)|_{Z}\right]/S^1
=\left[\mathbb{L}|_Z\boxtimes\mathbb{L_C}^\pm|_{{\{0\}}}\right]/S^1
$$
Since the $S^1$ action on the vector space
$\mathbb{L_C}^\pm|_{\{0\}}$ has weight $+1$ (see
Remark \ref{det_line_bndl}) we end up with
$$\mathbb{L}_{cut}^\pm|_{Z/S^1}=\mathbb{L}_{red}\otimes(N^-/S^1)=\mathbb{L}_{red}\otimes\overline{N}^{\;-}$$
as desired.

\end{proof}

\begin{cor}\label{cor-L_red}
If $F\subset Z/S^1\subset M_{cut}^\pm$ is a
connected component, then $S^1$ acts on the
fibers of $\left(\mathbb{L}_{cut}^\pm\right)|_F$
with weight $+1$.
\end{cor}
\begin{proof}
The previous lemma implies that
$$\mathbb{L}_{cut}^\pm|_F=\mathbb{L}_{red}|_F\otimes\overline{N}^{\;-}|_F .$$
The action of $S^1$ on $\mathbb{L}_{red}$ is
trivial. Using the isomorphism $\overline
N^{\;-}\simeq Z\times_{S^1}\mathbb{C}$ from
Corollary \ref{Normal bundle}, we see that the
action of $S^1$ on the fibers of
$\overline{N}^{\;-}|_F$ will have weight $+1$.

\end{proof}

\subsection{Third lemma - the spaces $M_\pm$} \ \\
Recall that $M\setminus Z=M_+\coprod M_-$
(disjoint union), where $M_\pm\subset M$ are open
submanifolds. We have embedding
$$i_\pm\colon M_\pm\to M_{cut}^\pm\qquad m\mapsto
[m,\sqrt{\pm\phi(m)}]$$ which are equivariant and
preserve the orientation (see Proposition 6.1 in
\cite{CKT}). Also recall that, as sets, we have
$M_{cut}^\pm=Z/S^1\coprod M_\pm$.

It is important to note that the embeddings
$M_\pm\to M_{cut}^\pm$ do not preserve the
metric. This, however, will not effect out
calculations.

\begin{lemma}\label{L-det-M-plus-minus}
The restriction of $\mathbb{L}$ to $M_\pm$ is
isomorphic to the restriction of $L_{cut}^\pm$ to
$M_\pm$. In other words,
$$ \mathbb{L}|_{M_\pm}\simeq \left(\mathbb{L}_{cut}^\pm\right)|_{M_\pm}$$
\end{lemma}

\begin{proof}
Let
$$\widetilde{M}_\pm=\left\{(m,\sqrt{\pm\phi(m)}):m\in
M_{\pm}\right\}\subset\widetilde{Z}^\pm\ ,$$ and
let
$$pr_1\colon M\times\mathbb{C}\to M\qquad,\qquad
pr_2\colon M\times\mathbb{C}\to\mathbb{C}$$ be
the projections. Then
$$\mathbb{L}_{cut}^\pm=\left[\left(pr_1^*(\mathbb{L})\otimes pr_2^*(\mathbb{L_C})\right)|_{\widetilde{Z}^\pm}\right]/S^1$$
and when restricting to $M_\pm$, we get
$$\mathbb{L}_{cut}^\pm|_{{M}_\pm}=pr_1^*(\mathbb{L})|_{\widetilde{M}_\pm}\otimes pr_2^*(\mathbb{L_C})|_{\widetilde{M}_\pm}
=\mathbb{L}|_{{M}_\pm}\otimes
pr_2^*(\mathbb{L_C})|_{\widetilde{M}_\pm}$$ Since
$M_\pm\simeq\widetilde M_\pm$. The term
$pr_2^*(\mathbb{L_C})|_{\widetilde{M}_\pm}$ is a
trivial equivariant complex line bundle, so we
conclude that
$$\mathbb{L}_{cut}^\pm|_{{M}_\pm}=\mathbb{L}|_{{M}_\pm}\otimes\mathbb{C}=\mathbb{L}|_{{M}_\pm}$$
as needed.
\end{proof}

\subsection{The proof of additivity under cutting}
Using all the preliminary lemmas, we can now
prove our main theorem.

\begin{proof}[Proof of Theorem \ref{mainthm}] \

\noindent Write $M\setminus Z=M_+\sqcup M_-$ .
Because the action $S^1\circlearrowright Z$ is
free, the submanifold $Z\subset M$ is a reducible
splitting hypersurface (see
\S\ref{spin-c-cutting}). Every connected
component $F\subset M^{S^1}$ of the fixed point
set must be a subset of either $M_+$ or $M_-$ .

\noindent Also recall that
$M_{cut}^\pm=M_\pm\sqcup Z/S^1$, and the action
of $S^1$ on $Z/S^1$ is trivial (and hence $Z/S^1$
is a subset of the fixed point set under the
action $S\circlearrowright M_{cut}^\pm$).

\noindent Using the Kostant formula (Theorem
\ref{Kos-noniso}) we get, for any weight
$\beta\in Lie(S^1)^*$,
\begin{multline*}
\#(\beta,Q(M))=\sum_{F\subset
M^{S^1}}(-1)^F\cdot\overline{N}_F\left(\beta-\frac{1}{2}\mu_F\right)\\
=\sum_{F\subset
\left(M_+\right)^{S^1}}(-1)^F\cdot\overline{N}_F\left(\beta-\frac{1}{2}\mu_F\right)
+\sum_{F\subset
\left(M_-\right)^{S^1}}(-1)^F\cdot\overline{N}_F\left(\beta-\frac{1}{2}\mu_F\right)\\
\end{multline*}
where the sum is taken over the connected
components of the fixed point sets. For the cut
spaces we have the following equalities.

\begin{multline*}
\#(\beta,Q(M_{cut}^\pm))=\sum_{F\subset
\left(M_{cut}^\pm\right)^{S^1}}(-1)^F\cdot\overline{N}_F\left(\beta-\frac{1}{2}\mu_F\right)\\
=\sum_{F\subset
\left(M_\pm\right)^{S^1}}(-1)^F\cdot\overline{N}_F\left(\beta-\frac{1}{2}\mu_F\right)
+\sum_{F\subset
Z/S^1}(-1)^F\cdot\overline{N}_F\left(\beta-\frac{1}{2}\mu_F\right)\ .\\
\end{multline*}

In order to prove additivity, we need to show
that

$$\sum_{F\subset Z/S^1\subset
M_{cut}^+}(-1)^F\cdot\overline{N}_F\left(\beta-\frac{1}{2}\mu_F\right)+\sum_{F\subset
Z/S^1\subset
M_{cut}^-}(-1)^F\cdot\overline{N}_F\left(\beta-\frac{1}{2}\mu_F\right)=0\
.\vspace{20pt}$$

\noindent Note that the summands in the two sums
above are different. In the first, we regard $F$
as a subset of $M_{cut}^+$, and in the second, as
a subset of $M_{cut}^-$.

Choose a connected component $F\subset Z/S^1$.
Note that $F$ is oriented by the reduced
orientation. Since $F$ can be regarded as a
subset of both $M_{cut}^+$ and $M_{cut}^-$, we
will add a superscript $F^\pm$ to emphasize that
$F$ is being thought of as a subspace of the
corresponding cut space.

It suffices to show that
\begin{equation*}\label{add_eqn}
(*)\qquad(-1)^{F^+}\cdot\overline{N}_{F^+}\left(\beta-\frac{1}{2}\mu_{F^+}\right)+(-1)^{F^-}\cdot\overline{N}_{F^-}\left(\beta-\frac{1}{2}\mu_{F^-}\right)=0\vspace{20pt}
\end{equation*}

Recall that $Z\subset M$ is of (real) codimension
$1$, and so $Z/S^1\subset M_{cut}^\pm$ is of
(real) codimension 2. Therefore, the normal
bundle $NF^\pm$ to $Z/S^1$ in the cut spaces has
rank 2. We can turn the bundles $NF^\pm$ to
complex line bundles using Corollary \ref{Normal
bundle}, and then the weight of the action
$S^1\circlearrowright NF^\pm$ will be $-1$ for
$NF^+$ and $+1$ for $NF^-$.

This is, however, not good, since in order to
write down Kostant's formula, we need our weights
to be polarized. Therefore, we will use for
$NF^{\;-}$ the complex structure coming from the
isomorphism
$$NF^{\;-}\xrightarrow{\ \simeq\ }
Z\times_{S^1}\mathbb{C}\ ,$$ and for $NF^+$, we
will use the complex structure which is
\emph{opposite} to the one induced by the
isomorphism $$NF^{+}\xrightarrow{\ \simeq\ }
Z\times_{S^1}\mathbb{C}\ .$$

With this convention, the bundles $NF^\pm$ become
isomorphic as equivariant complex line bundles,
and the weight of the $S^1$-action on those
bundles is $+1$.\\Also, Lemma \ref{det-reduced}
implies that the determinant line bundles
$\mathbb{L}_{cut}^\pm$, when restricted to $F$,
are isomorphic as equivariant complex line
bundles, and the weight of the $S^1$-action on
the fibers of $\mathbb{L}_{cut}^\pm|_F$ is $+1$.

Recall now (see \S5.3) that the explicit
expression for
$\displaystyle\overline{N}_{F^\pm}\left(\beta-\frac{1}{2}\mu_{F^\pm}\right)$
involves only the following ingredients:
\begin{itemize}
\item $\mu_{F^\pm}$, which are equal to each other
($\mu_{F^\pm}=+1$), since
$\mathbb{L}_{cut}^+|_F\simeq~
\mathbb{L}_{cut}^-|_F$.
\item $c_1(NF^\pm)$, which are  equal since $NF^\pm$ are isomorphic as
complex line bundle, by our previous remark.
\item $\hat A(TF)$, which are equal, since $F^+=F^-$
as manifolds.

This means that the terms $\overline{N}_{F^\pm}$
in equation (*) above are the same.
\end{itemize}

So all is left is to explain why
$$(-1)^{F^+}+(-1)^{F^-}=0\ .$$
But this follows easily from Claim \ref{orien -
normal}. This claim implies that the orientation
on $F^-$, followed by the one of $NF^{\;-}$,
gives the orientation of $M_{cut}^-$. Hence,
$(-1)^{F^-}=1$. Since we switched the original
orientation for $NF^+$, composing the orientation
of $F^+$ with the one of $NF^+$ will give the
opposite orientation on $M_{cut}^+$, and hence
$(-1)^{F^+}=-1$. The additivity result follows.

\end{proof}

\section{An example: the
two-sphere}\label{two-sphere}

In this section we give an example, which
illustrates the additivity of spin$^c$
quantization under cutting.

In this example, the manifold is the standard
two-sphere $M=S^2\subset\mathbb{R}^3$, with the
outward orientation and the standard Riemannian
structure. The circle group
$S^1\subset\mathbb{C}$ acts effectively on the
two sphere by rotations about the z-axis.

We will need the following lemma.
\begin{lemma}
Let $M$ be an oriented Riemannian manifold, on
which a Lie group $G$ acts transitively by
orientation preserving isometries. Choose a point
$x\in M$ and denote by $G_x$ the stabilizer at
$x$ and by $\sigma\colon G_x\to SO(T_x M)$ the
isotropy representation. Then:
\begin{enumerate}
\item $G_x$ acts on $SO(T_xM)$ by $g\cdot
A=\sigma(g)\circ A$ .
\item The map $$G\to M\qquad,\qquad g\mapsto
g\cdot x$$ is a principal $G_x$-bundle (where
$G_x$ acts on $G$ by right multiplication).
\item The principal $SO(T_xM)$-bundle
$G\times_{G_x}SO(T_xM)$ is isomorphic to
$SOF(M)$, the bundle of oriented orthonormal
frames on $M$.
\end{enumerate}
\end{lemma}
\begin{proof}
(1) is easy. (2) follows from Proposition B.18 in
\cite{Kar} (with $H=G_x$), together with the fact
that $G/G_x$ is diffeomorphic to $M$. To show
(3), consider the map taking an element $[g,A]\in
G\times_{G_x}SO(T_xM)$ to the frame $g_*\circ
A\colon T_xM\xrightarrow{\simeq} T_{g\cdot x}M$.
This map can be easily checked to be an
isomorphism of principal $SO(T_xM)$-bundles.
\end{proof}

\subsection{The trivial $S^1$-equivariant spin$^c$ structure
on $S^2$} \

To define an $S^1$-equivariant spin$^c$ structure
on $S^2$, one needs to describe the space $P$ and
the maps in a commutative diagram of the
following form (see Remark
\ref{Remark_spin-c_str}).

$$
\begin{CD}
S^1\times P       @>>>     P    @<<<     P\times Spin^c(2)\\
@VVV                    @V\Lambda VV          @VVV\\
S^1\times SOF(S^2)       @>>>     SOF(S^2)    @<<<     SOF(S^2)\times SO(2)\\
@VVV                   @V\pi VV           @.\\
S^1\times S^2       @>>>     S^2    @.\\
\end{CD}
$$\\

Set $P=Spin^c(3)$. By  the above lemma, the
choice of a point $x=(0,0,1)\in S^2$ and a basis
for $T_xS^2$ give an isomorphism between the
frame bundle of $S^2$ and
$SO(3)\times_{SO(2)}SO(2)=SO(3)$. Thus $SOF(S^2)
\cong SO(3)$, and our diagram becomes

$$
\begin{CD}
S^1\times Spin^c(3)       @>>>     Spin^c(3)    @<<<     Spin^c(3)\times Spin^c(2)\\
@VVV                    @V\Lambda VV          @VVV\\
S^1\times SO(3)       @>>>     SO(3)    @<<<     SO(3)\times SO(2)\\
@VVV                   @V\pi VV           @.\\
S^1\times S^2       @>>>     S^2    @.\\
\end{CD}
$$\\

Now we describe the maps in this diagram. Denote
\[C_\theta=\left(
                    \begin{array}{ccc}
                    \cos\theta & -\sin\theta & 0\\
                    \sin\theta & \cos\theta  & 0\\
                    0          &     0       & 1
                    \end{array}
\right)\]

\noindent The map $S^1\times S^2\to S^2$ is
rotation about the vertical axis, i.e., $
(e^{i\theta},v)\mapsto C_\theta\cdot v$ .

The second horizontal row gives the actions of
$S^1$ and $SO(2)$ on the frame bundle $SO(3)$.
Those are  given by left and right multiplication
by $C_\theta$, respectively. The covering map
$\pi\colon SO(3)\to S^2$ is given by $A\mapsto
A\cdot x$, and $\Lambda$ is the natural map from
the spin$^c$ group to the special orthogonal
group.

All is left is to describe the actions of $S^1$
and $Spin^c(2)$ on $Spin^c(3)$ (the top row in
the diagram). Since $Spin^c(2)\subset Spin^c(3)$,
this group will act by right-multiplication. The
$S^1$-action on $Spin^c(3)$ is given by
\begin{equation}\label{action}
(e^{i\theta}\;,\;[A,z])\mapsto [x_{\theta/2}\cdot
A\;,\;e^{i\theta/2}\cdot z] \end{equation}
where
$x_\theta=\cos\theta+\sin\theta\cdot e_1e_2\in
Spin(3)$. Note that $x_{\theta/2}$ and
$e^{i\theta/2}$ are defined only up to sign, but
the equivalence class
$[x_{\theta/2}\;,\;e^{i\theta/2}]$ is a well
defined element in $Spin^c(3)$.

We will call this $S^1$-equivariant spin$^c$
structure \emph{the trivial spin$^c$ structure on
the $S^1$-manifold $S^2$}, and denote it by
$P_0$. The reason for using the word `trivial' is
justified by the following lemma.

\begin{lemma}
The determinant line bundle of the trivial
spin$^c$ structure $P_0$ is isomorphic to the
trivial complex line bundle $\mathbb{L}\cong
S^2\times\mathbb{C}$, with the non-trivial
$S^1$-action
$$ S^1\times\mathbb{L}\to\mathbb{L}\qquad,\qquad
(e^{i\theta},(v,z))\mapsto (C_\theta\cdot v,
e^{i\theta}\cdot z)$$
\end{lemma}

\begin{proof}
It is easy to check that the
map
$$\mathbb{L}=Spin^c(3)\times_{Spin^c(2)}\mathbb{C}\to
 S^2\times\mathbb{C}\qquad,\qquad
 [[A,z],w]\mapsto (\lambda(A)\cdot x,z^2w)\ ,$$
where $\lambda\colon Spin(3)\to SO(3)$ is the
double cover and $x=(0,0,1)$ is the north pole,
is an isomorphism of complex line bundles. The
fact that $S^1$ acts on $\mathbb{L}$ via
(\ref{action}), and that
$\lambda(x_{\theta/2})=C_\theta$, implies that
the $S^1$ action on $S^2\times\mathbb{C}$,
induced by the above isomorphism, is the one
stated in the lemma.
\end{proof}

Another reason for calling $P_0$ a trivial
spin$^c$ structure, is that the quantization
$Q(S^2)$ (with respect to $P_0$) is the zero
space. We do not prove this fact now, since it
will follow from a more general statement (see
Claim \ref{Quan_S^2}).

\subsection{Classifying all spin$^c$ structures
on $S^2$.} \

Quantizing the trivial spin$^c$ structure on
$S^2$ is not interesting, since the quantization
is the zero space. However, once we have an
equivariant spin$^c$ structure on a manifold, we
can generate all the other equivariant spin$^c$
structures by twisting it with complex
equivariant Hermitian line bundles (or,
equivalently, with equivariant principal
$U(1)$-bundles). For details on this process, see
Appendix D, \S 2.7 in \cite{Kar}. We will use
this technique to construct all spin$^c$
structures on our $S^1$-manifold $S^2$.

It is known that all (non-equivariant) complex
Hermitian line bundles over $S^2$ are classified
by $H^2(S^2;\mathbb{Z})\cong\mathbb{Z}$, i.e., by
the integers. The $S^1$-equivariant line bundles
over $S^2$ are classified by a pair of integers
(for instance, the weights of the $S^1$-action on
the fibers at the poles). This is well known, but
because we couldn't find a direct reference, we
will give a direct proof of this fact.\\Here is
an explicit construction of an equivariant line
bundle over $S^2$, determined by a pair of
integer.

\begin{definition}
Given a pair of integers $(k,n)$, define an
$S^1$-equivariant complex Hermitian line bundle
$L_{k,n}$ as follows:
\begin{enumerate}
\item As a complex line bundle, $$L_{k,n}=Spin(3)\times_{Spin(2)}\mathbb{C}\cong
S^3\times_{S^1}\mathbb{C}\ ,$$ where
$Spin(2)\cong S^1$ acts on $\mathbb{C}$ with
weight $n$ and on $Spin(3)$ by right
multiplication.
\item The circle group $S^1$ acts on $L_{k,n}$ by
$$ S^1\times L_{k,n}\to L_{k,n}\qquad,\qquad
\left( e^{i\theta},[A,z]\right)\mapsto
[x_{\theta/2}\cdot
A,e^{\frac{i\theta}{2}(n+2k)}\cdot z]$$ where
$x_\theta=\cos\theta+\sin\theta\cdot e_1e_2\in
Spin(2)\subset Spin(3)$.
\end{enumerate}
\end{definition}

And now we prove:

\begin{claim}
Every $S^1$ equivariant line bundle over $S^2$ is
isomorphic to $L_{k,n}$, for some
$k,n\in\mathbb{Z}$.
\end{claim}

\begin{proof}
Let $L$ be an $S^1$-equivariant line bundle over
$S^2$. Since $L$ is, in particular, an ordinary
line bundle, we can assume it is of the form
$L=S^3\times_{S^1}\mathbb{C}$ where $S^1$ acts on
$\mathbb{C}$ with weight $n$. Also, since $L$ is
an equivariant line bundle, we have a map
$$\rho\colon S^1\times L\to L\qquad,\qquad
(e^{i\theta},x)\mapsto e^{i\theta}\cdot x\ .$$

Define a map $$\eta\colon S^1\times L\to
L\qquad,\qquad (e^{i\theta},[A,z])\mapsto
[x_{-\theta/2}\cdot A,e^{-i\theta n/2}z]\ .$$
This map is well defined.

By composing $\rho$ and $\eta$ we get a third map
$$ \delta\colon S^1\times L\to L$$
which lifts the trivial action on $S^2$. Since
$S^2$ is connected, this composed action will act
on all the  fibers of $L$ with one fixed weight
$k$. Therefore, we get
$$e^{i\theta}\cdot [x_{-\theta/2}\cdot A,e^{-i\theta
n/2}z]=[A,e^{ik\theta}z]$$

and after a change of variables, the given action
$S^1\circlearrowright L$ is
$$e^{i\theta}\cdot[B,w]=[x_{\theta/2}\cdot
B,e^{i\theta n/2+ik\theta}w]\ .$$ This means that
$L$ is isomorphic to $L_{k,n}$.
\end{proof}

We now `twist' the trivial spin$^c$ structure by
$U(L_{k,n})$, the unit circle bundle of
$L_{k,n}$, to get nontrivial spin$^c$ structures
on $S^2$. Observe that the group $U(1)$ acts on
$Spin^c(3)$ from the right by multiplication by
elements of the form $[1,c]\in Spin^c(3)$.

\begin{definition}
$$P_{k,n}=P_0\times_{U(1)}
U(L_{k,n})$$ where we quotient by the
anti-diagonal action of $U(1)$.
\end{definition}

This is an $S^1$-equivariant spin$^c$ structure
on $S^2$. The principal action of $ Spin^c(2)$
comes from acting from the right on the $P_0\cong
Spin^c(3)$ component, and the left $S^1$-action
is induced from the diagonal action on $P_0\times
L_{k,n}$.

\begin{claim}
Fix $(k,n)\in\mathbb{Z}^2$, and denote by
$\mathbb{L}=\mathbb{L}_{k,n}$ the determinant
line bundle associated to the spin$^c$ structure
$P_{k,n}$ on $S^2$. Let $N=(0,0,1)\ ,\
S=(0,0,-1)\in S^2$ be the north and the south
poles.\\ Then $S^1$ acts on $\mathbb{L}|_N$ with
weight $2k+2n+1$ and on $\mathbb{L}|_S$ with
weight $2k+1$.
\end{claim}

\begin{proof}
The determinant line bundle is
$$\mathbb{L}=P_{k,n}\times_{Spin^c(2)}\mathbb{C}=\left[Spin^c(3)\times_{U(1)}\left(S^3\times_{S^1}S^1\right)\right]
\times_{Spin^c(2)}\mathbb{C}\ .$$ An element of $\mathbb{L}$ can be
written in the form $\left[\left[[A,1],[A,1]\right],u\right]$, where
$A\in Spin(3)\cong S^3$ and $u\in\mathbb{C}$.
\begin{enumerate}
\item For the north pole $N=(0,0,1)$, can choose $A=1\in Spin(3)$,
hence an element of $\mathbb{L}|_N$ will have the form
$\left[\left[[1,1],[1,1]\right],u\right]$. Let $e^{i\theta}\in S^1$,
act on $\mathbb{L}|_N$, to get
\begin{multline*}
\left[\left[[x_{\theta/2},e^{i\theta/2}],[x_{
\theta/2},e^{i\theta(n+2k)/2}]\right],u\right]=\left[\left[[1,1],[x_{
\theta/2},e^{i\theta(n+2k)/2}]\right],e^{i\theta}u\right]=\\
\quad=\left[\left[[1,1],[x_{ \theta/2},e^{i\theta\cdot
n/2}]\right],e^{i\theta(1+2k)}u\right]=
\left[\left[[1,1],[1,e^{i\theta\cdot n/2}\cdot e^{i\theta\cdot
n/2}]\right],e^{i\theta(1+2k)}u\right]=\\
\quad=\left[\left[[1,1],[1,1]\right],e^{i\theta(1+2k+2n)}u\right]\hfill\\
\end{multline*}
and therefore the weight on $\mathbb{L}|_N$ is $1+2k+2n$.
\item For the south pole $S=(0,0,-1)$ choose $A=e_2e_3$. We
compute again the action of an element $e^{i\theta}$ on
$\left[\left[[A,1],[A,1]\right],u\right]$, and use the identity
$A\cdot x_\theta=x_{-\theta}\cdot A$ for any $x_\theta\in
Spin(2)\subset Spin(3)$.
\begin{multline*}
\left[\left[[x_{\theta/2} A,e^{i\theta/2}],[x_{ \theta/2}
A,e^{i\theta(n+2k)/2}]\right],u\right]=\left[\left[[A ,1],[Ax_{-
\theta/2},e^{i\theta(n+2k)/2}]\right],e^{i\theta}u\right]=\\
\quad=\left[\left[[A,1],[Ax_{-\theta/2},e^{i\theta\cdot
n/2}]\right],e^{i\theta(1+2k)}u\right]=
\left[\left[[A,1],[A,e^{-i\theta\cdot n/2}\cdot e^{i\theta\cdot
n/2}]\right],e^{i\theta(1+2k)}u\right]=\\
\quad=\left[\left[[A,1],[A,1]\right],e^{i\theta(1+2k)}u\right]\hfill\\
\end{multline*}
and therefore the weight on $\mathbb{L}|_S$ is $2k+1$.
\end{enumerate}
\end{proof}

\begin{remark}
Note that the $2k+2n+1$ and $2k+1$ are both odd numbers. This is not
surprising in view of Lemma \ref{weight-lemma}. The isotropy weight
at $N$ (or at $S$) is $\pm1$ and its sum with the weight on
$\mathbb{L}_N$  (or on $\mathbb{L}_S$) must be even. This implies
that the weights of $S^1\circlearrowright\mathbb{L}_{\{N,S\}}$ must
be odd.
\end{remark}
\begin{remark}
The above claim implies that the determinant line
bundle of the spin$^c$ structure $P_{k,n}$ is
isomorphic to $L_{2k+1,2n}$, i.e., $\mathbb
L_{k,n}\cong L_{2k+1,2n}$ .
\end{remark}

\begin{claim} \label{Quan_S^2}
Fix $(k,n)\in\mathbb{Z}^2$ and denote by
$Q_{k,n}(S^2)$ the quantization of the spin$^c$
structure $P_{k,n}$ on $S^2$. Then the
multiplicity of a weight $\beta\in
Lie(S^1)^*\cong\mathbb{Z}$ in $Q_{k,n}(S^2)$ is
given by
$$\#(\beta,Q_{k,n}(S^2))=
    \begin{cases}
    \ 1   &\quad 0<\beta-k\le n\\
    -1    &\quad n<\beta-k\le 0\\
    \  0  &\quad \text{\ \ otherwise}\\
    \end{cases}$$
    In particular, if $n=0$, then $Q_{k,0}(S^2)$ is
    the zero representation.
\end{claim}

\begin{proof}
By the Kostant formula for spin$^c$-quantization
(Theorem \ref{Kost}) the multiplicity is given by
$$\#(\beta,Q_{k,n}(S^2))=\overline
N_{(0,0,1)}\left(\beta-\frac{1+2k+2n}{2}\right)-\overline
N_{(0,0,-1)}\left(\beta-\frac{1+2k}{2}\right)\
.$$ The definition of $\overline N_p$ implies
that
$$\overline N_{(0,0,1)}\left(\beta-\frac{1+2k+2n}{2}\right)=
\begin{cases}
1 & \quad\beta-k\le n\\
0 & \quad\beta-k>n
\end{cases}$$
and similarly,
$$\overline N_{(0,0,-1)}\left(\beta-\frac{1+2k}{2}\right)=
\begin{cases}
1 & \quad\beta-k\le 0\\
0 & \quad\beta-k>0
\end{cases}\ .$$
Using that, one can compute $\#(\beta,Q_{k,n}(S^2))$ and get the
required result.
\end{proof}

\subsection{Cutting a spin$^c$ structure on $S^2$}
\

Now we get to the cutting of the spin$^c$
structure $P_{k,n}$ on $S^2$. Let $\mathbb{L}$ be
the determinant line bundle of $P_{k,n}$. We take
the equator
$Z=\{(\cos\alpha,\sin\alpha,0)\}\subset S^2$ to
be our reducible splitting hypersurface (see
\S\ref{spin-c-cutting}). The cut spaces
$M_{cut}^\pm$ are both diffeomorphic to $S^2$,
and we would like to know what are
$(P_{k,n})_{cut}^\pm$. Because the cut spaces are
spheres again, we must have
$$(P_{k,n})_{cut}^\pm=P_{k^\pm,n^\pm}\qquad\text{for some
integers}\qquad k^\pm,n^\pm\ .$$ Corollary
\ref{cor-L_red} implies that $S^1$ acts on
$\mathbb{L}_{cut}^-|_N$ and on
$\mathbb{L}_{cut}^+|_S$ with weight $+1$.  Lemma
\ref{L-det-M-plus-minus} implies that the weight
of the $S^1$ action on $\mathbb{L}|_N$ and
$\mathbb{L}|_S$ will be equal to the weight of
the action on $\mathbb{L}_{cut}^+|_N$ and
$\mathbb{L}_{cut}^-|_S$, respectively. From this
we get the equations $$2k^++1=1\quad,\quad
2k^++2n^++1=2k+2n+1\quad,\quad
2k^-+2n^-+1=1\quad,\quad 2k^-+1=2k+1$$ which
yield $k^+=0,\ n^+=k+n,\ k^-=k,\ n^-=-k$.
Therefore we obtain:
$$(P_{k,n})_{cut}^+=P_{0,k+n}\qquad,\qquad
(P_{k,n})_{cut}^-=P_{k,-k}\ .$$

\begin{remark}
We see that there is no symmetry between the
spin$^c$ structures on the `$+$' and `$-$' cut
spaces as one might expect. This is because the
definition of the covering map $SO(3)\to S^2$
involved a choice of a point (in our case - the
north pole), which `broke' the symmetry of the
two-sphere.
\end{remark}

The quantization of the cut spaces is thus
obtained from Claim \ref{Quan_S^2}. For the `+'
cut space we get, for any weight
$\beta\in\mathbb{Z}$:
$$\#(\beta,Q_{k,n}^+(S^2))=\#(\beta,Q_{0,k+n}(S^2))=
    \begin{cases}
    \ 1   &\quad -k<\beta-k\le n\\
    -1    &\quad n<\beta-k\le -k\\
    \  0  &\quad \text{\ \ \ \ otherwise}\\
    \end{cases}$$

and for the `$-$' cut space:

$$\#(\beta,Q_{k,n}^-(S^2))=\#(\beta,Q_{k,-k}(S^2))=
    \begin{cases}
    \ 1   &\quad 0<\beta-k\le -k\\
    -1    &\quad -k<\beta-k\le 0\\
    \  0  &\quad \text{\ \ \ \ otherwise}\\
    \end{cases}$$

It is an easy exercise to check that
$$\#(\beta,Q_{k,n}(S^2))=\#(\beta,Q_{k,n}^+(S^2))+\#(\beta,Q_{k,n}^-(S^2))$$
and this implies that as virtual
$S^1$-representations, we have
$$Q_{k,n}(S^2)=Q_{k,n}^-(S^2)\oplus
Q_{k,n}^+(S^2)$$

As expected, we have additivity of spin$^c$
quantization under cutting in this example.

\subsection{Multiplicity Diagrams}\ \\
The $S^1$-equivariant spin$^c$ quantization of a
manifold $M$ can be described using multiplicity
diagrams as follows. Above each integer on the
real line, we write the multiplicity of the
weight represented by this integer, if it is
nonzero.

For example, if $n,k>0$, then the quantization
$Q_{k,n}$ of $S^2$ is given by the following diagram.\\

{\small \hspace{121pt} +1 \hspace{7pt} +1
\hspace{40pt}$\cdots$\hspace{30pt} +1
\hspace{7pt} +1 \hspace{7pt}
+1\\[-5pt]
\line(1,0){20}\put(0,-2){$\shortmid$}
\line(1,0){95}\put(0,-2){$\shortmid$}
\line(1,0){25}\put(0,-2.5){$\bullet$}
\line(1,0){25}\put(0,-2.5){$\bullet$}
\line(1,0){100}\put(0,-2.5){$\bullet$}
\line(1,0){25}\put(0,-2.5){$\bullet$}
\line(1,0){25}\put(0,-2.5){$\bullet$}
\line(1,0){25}\put(0,-2){$\shortmid$}
\line(1,0){25}\put(-1,-2.3){$\rightarrow$}\\
$\phantom{\ }\hspace{15.5pt}0\hspace{90.5pt} k
\hspace{15pt}k\negmedspace+\negmedspace1
\hspace{158pt}n\negmedspace+\negmedspace k$\\[20pt]}

The quantization of the `+' cut space,
$Q_{k,n}^+$, which is equal to $Q_{0,k+n}$, will
have the following diagram.\\[10pt]

{\small
\hspace{29pt}+1\hspace{25pt}$\cdots$\hspace{22pt}+1\hspace{8pt}
+1 \hspace{7pt} +1
\hspace{40pt}$\cdots$\hspace{30pt} +1
\hspace{7pt} +1 \hspace{7pt}
+1\\[-5pt]
\line(1,0){20}\put(0,-2){$\shortmid$}
\line(1,0){25}\put(0,-2.5){$\bullet$}
\line(1,0){70}\put(0,-2.5){$\bullet$}
\line(1,0){25}\put(0,-2.5){$\bullet$}
\line(1,0){25}\put(0,-2.5){$\bullet$}
\line(1,0){100}\put(0,-2.5){$\bullet$}
\line(1,0){25}\put(0,-2.5){$\bullet$}
\line(1,0){25}\put(0,-2.5){$\bullet$}
\line(1,0){25}\put(0,-2){$\shortmid$}
\line(1,0){25}\put(-1,-2.3){$\rightarrow$}\\
$\phantom{\
}\hspace{15.5pt}0\hspace{22pt}1\hspace{65pt} k
\hspace{14pt}k\negmedspace+\negmedspace1
\hspace{158pt}n\negmedspace+\negmedspace
k$\\[20pt]
}

Finally, $Q_{k,n}^-=Q_{k,-k}$ is given
by\\[10pt]

{\small
\hspace{29pt}$-1$\hspace{25pt}$\cdots$\hspace{22pt}$-1$\\[-5pt]
\line(1,0){20}\put(0,-2){$\shortmid$}
\line(1,0){25}\put(0,-2.5){$\bullet$}
\line(1,0){70}\put(0,-2.5){$\bullet$}
\line(1,0){25}\put(0,-2){$\shortmid$}
\line(1,0){225}\put(-1,-2.3){$\rightarrow$}\\
$\phantom{\
}\hspace{15.5pt}0\hspace{22pt}1\hspace{65pt} k$
\\[20pt]
} Clearly, one can see that the diagram of
$Q_{k,n}$ is the `sum' of the diagrams of
$Q_{k,n}^\pm$.

Let us present another case, where only positive
multiplicities occur in the quantization of all
three spaces (the original manifold $S^2$ and the
cut spaces). This happens if $k<0<n+k$. In this
case, the diagram for $Q_{k,n}$ is as
follows.\\[10pt]

{\small
\hspace{29pt}+1\hspace{25pt}$\cdots$\hspace{22pt}+1\hspace{8pt}
+1 \hspace{7pt} +1
\hspace{40pt}$\cdots$\hspace{30pt} +1
\hspace{7pt} +1 \hspace{7pt}
+1\\[-5pt]
\line(1,0){20}\put(0,-2){$\shortmid$}
\line(1,0){25}\put(0,-2.5){$\bullet$}
\line(1,0){70}\put(0,-2.5){$\bullet$}
\line(1,0){25}\put(0,-2.5){$\bullet$}
\line(1,0){25}\put(0,-2.5){$\bullet$}
\line(1,0){100}\put(0,-2.5){$\bullet$}
\line(1,0){25}\put(0,-2.5){$\bullet$}
\line(1,0){25}\put(0,-2.5){$\bullet$}
\line(1,0){25}\put(0,-2){$\shortmid$}
\line(1,0){25}\put(-1,-2.3){$\rightarrow$}\\
$\phantom{\
}\hspace{15.5pt}k\hspace{15pt}k\negmedspace+\negmedspace1\hspace{85pt}
0 \hspace{20.5pt}1
\hspace{138pt}n\negmedspace+\negmedspace
k$\\[20pt]
} The diagram for $Q_{k,n}^+=Q_{0,n+k}$
is\\[10pt]

{\small \hspace{147pt} +1
\hspace{40pt}$\cdots$\hspace{30pt} +1
\hspace{7pt} +1 \hspace{7pt}
+1\\[-5pt]
\line(1,0){20}\put(0,-2){$\shortmid$}
\line(1,0){95}
\line(1,0){25}\put(0,-2.5){$\shortmid$}
\line(1,0){25}\put(0,-2.5){$\bullet$}
\line(1,0){100}\put(0,-2.5){$\bullet$}
\line(1,0){25}\put(0,-2.5){$\bullet$}
\line(1,0){25}\put(0,-2.5){$\bullet$}
\line(1,0){25}\put(0,-2){$\shortmid$}
\line(1,0){25}\put(-1,-2.3){$\rightarrow$}\\
$\phantom{\ }\hspace{15.5pt}k\hspace{115.5pt} 0
\hspace{22pt}1
\hspace{139pt}n\negmedspace+\negmedspace k$\\[20pt]}

and for $Q_{k,n}^-=Q_{k,-k}$ we have\\[10pt]

{\small
\hspace{29pt}$+1$\hspace{25pt}$\cdots$\hspace{22pt}$+1$\hspace{13pt}$+1$\\[-5pt]
\line(1,0){20}\put(0,-2){$\shortmid$}
\line(1,0){25}\put(0,-2.5){$\bullet$}
\line(1,0){70}\put(0,-2.5){$\bullet$}
\line(1,0){25}\put(0,-2.5){$\bullet$}
\line(1,0){27}\put(0,-2){$\shortmid$}
\line(1,0){198}\put(-1,-2.3){$\rightarrow$}\\
$\phantom{\ }\hspace{15.5pt}k\hspace{117pt} 0$
\\[20pt]
}

\noindent and again the additivity is clear.

The additivity is clearer in the last set of
diagrams, as we can actually see the diagram of
$Q_{k,n}$ being cut into two parts. It seems like
the diagram was cut at some point between 0 and
1. The point at which the cutting is done depends
on the spin$^c$ structure on $\mathbb{C}$ that
was chosen during the cutting process (see
\S\ref{spin-c-cutting}).

\section{Relation to symplectic cutting}
The cutting construction was originally defined
for symplectic manifolds (see \cite{L}). In this
paper we followed \cite{CKT} and defined cutting
for manifold which are not necessarily
symplectic. However, our work can be related to
symplectic cutting as described in \cite{Me}. We
outline the main ideas of this procedure.

Assume that a symplectic manifold $(M,\omega)$ is
endowed with a spin$^c$ structure $P\to SOF(M)\to
M$ (with respect to an orientation and a
Riemannian structure). When a connection 1-form
$\theta\in\Omega^1(P;\mathfrak u(1))$ on $P\to
SOF(M)$ is given, then the following
compatibility condition between the symplectic
structure, the spin$^c$ structure, and the
connection may be imposed:
$$d\theta=\pi^*(-i\cdot\omega)\ ,$$
where $\pi\colon P\to M$ is the projection. When
this condition (and one more technical
assumption) are satisfied, then $(P,\theta)$ is
called \emph{a spin$^c$ prequantization} for
$(M,\omega)$ (alternatively, given an oriented
Riemannian manifold $M$, we `prequantize' the
manifold $(M,\omega)$, where $\omega$ is a closed
two-form, determined by the above equality). If
all those structures respect a $G$-action on the
bundles $P\to SOF(M)\to M$, then  $G$ acts on $M$
in a Hamiltonian fashion, with a `natural' moment
map $\Phi\colon M\to\mathfrak{g}^*$ given by
$$\pi^*\left(\Phi^\xi\right)=-i\left(\iota_{\xi_P}\theta\right)\qquad,\qquad\xi\in\mathfrak{g}\ .$$

In the case where $G=S^1$, we can cut the
manifold $M$ along a level set of the moment map
$\Phi$. The cutting construction can be extended
to the spin$^c$ prequantization $(P,\theta)$, and
so we end up with two pairs
$(P^\pm_{cut},\theta^\pm_{cut})$ for the cut
spaces $(M_{cut}^\pm,\omega^\pm_{cut})$.

The cutting construction for spin$^c$
prequantization involves a choice of an odd
integer $\ell\in\mathbb Z$. It turns out that,
$(P^\pm_{cut},\theta_{cut}^\pm)$ are spin$^c$
prequantizations for
$(M_{cut}^\pm,\omega^\pm_{cut})$ if and only if
the cutting was done alone the submanifold
$Z=\Phi^{-1}(\ell/2)$.

On the level of multiplicity diagrams, the
diagram of $P$ will be cut at $\ell/2$ to give
the diagrams for $P_{cut}^\pm$. The fact that
$\ell/2$ is not an integer is the reason for
having additivity.

Details about spin$^c$ prequantization for
symplectic manifolds and the corresponding
cutting construction will be available in
\cite{Me}.


\end{document}